\documentclass[11pt, reqno]{amsart}
\usepackage[usenames,dvipsnames]{xcolor}
\definecolor{darkblue}{rgb}{0,0,.75}

\usepackage{amsfonts,amssymb,enumerate,bm,hhline,makecell,array,mathtools}
\usepackage{mathrsfs}
\usepackage{comment}
\usepackage[normalem]{ulem}
\usepackage[colorlinks=true, citecolor=darkblue, urlcolor=darkblue, linkcolor=darkblue, linktocpage = true]{hyperref}
\usepackage{amscd}
\usepackage[shortlabels]{enumitem}
\setlist[itemize]{topsep=5pt,itemsep=3pt}
\setlist[enumerate]{topsep=5pt,itemsep=3pt}
\usepackage{tensor}

\usepackage{tikz}
\usetikzlibrary{calc,trees,arrows,chains,shapes.geometric,%
	decorations.pathreplacing,decorations.pathmorphing,shapes,%
	matrix,shapes.symbols}
\tikzset{
	>=stealth',
	punktchain/.style={
		rectangle,
		rounded corners,
		draw=black, thick,
		minimum height=3em,
		text centered,
		on chain},
	line/.style={draw, thick, <-},
	element/.style={
		tape,
		top color=white,
		bottom color=blue!50!black!60!,
		minimum width=8em,
		draw=blue!40!black!90, very thick,
		text width=10em,
		minimum height=3.5em,
		text centered,
		on chain},
	every join/.style={->, thick,shorten >=1pt},
	decoration={brace},
	tuborg/.style={decorate},
	tubnode/.style={midway, right=2pt},
}
\usepackage{tikz-cd}
\usetikzlibrary{calc,matrix,patterns,shadings,backgrounds,fit}
\pgfdeclarelayer{myback}
\pgfsetlayers{myback,background,main}

\usepackage{extarrows}
\usepackage[new]{old-arrows}
\usepackage[toc,page]{appendix}
\usepackage[all,ps,cmtip,rotate]{xy}

\allowdisplaybreaks

\usepackage[top=3.75cm, bottom=3cm, left=3cm, right=3cm]{geometry}
\newcommand{\rK}{\mathrm{K}}

\def\fS{\mathfrak{S}}

\def\cD{\mathscr{D}}
\def\cA{\mathscr{A}}

\def\cF{\mathscr{F}}

\def\cO{\mathscr{O}}

\def\cP{\mathscr{P}}
\def\cH{\mathscr{H}}

\def\cT{\mathscr{T}}

\def\C{\ensuremath{\mathbb{C}}}

\def\H{\ensuremath{\mathbb{H}}}

\def\Q{\ensuremath{\mathbb{Q}}}
\def\R{\ensuremath{\mathbb{R}}}
\def\Z{\ensuremath{\mathbb{Z}}}

\def\bv{\ensuremath{\mathbf{v}}}

\DeclareMathOperator{\Aut}{Aut}

\DeclareMathOperator{\ch}{ch}

\DeclareMathOperator{\Coh}{Coh}

\newcommand{\Db}{\mathrm{D}^{\mathrm{b}}}
\newcommand{\Dqc}{\mathrm{D_{qc}}}

\DeclareMathOperator{\Hilb}{Hilb}
\DeclareMathOperator{\Hom}{Hom}
\DeclareMathOperator{\id}{id}

\def\Im{\mathop{\rm Im}\nolimits}
\DeclareMathOperator{\Ker}{Ker}

\DeclareMathOperator{\Kum}{Kum}
\def\Jac{\mathrm{Jac}}

\DeclareMathOperator{\Pic}{Pic}

\def\Re{\mathop{\rm Re}\nolimits}

\DeclareMathOperator{\rk}{rk}

\DeclareMathOperator{\Supp}{Supp}

\DeclareMathOperator{\Spec}{Spec}
\DeclareMathOperator{\Stab}{Stab}
\DeclareMathOperator{\rank}{rank}

\DeclareMathOperator{\td}{td}

\newcommand{\st}{\mid} 

\DeclareMathOperator{\Knum}{\rK_{\mathrm{num}}}

\numberwithin{equation}{section}

\theoremstyle{plain}
\newtheorem{lemm}{Lemma}[section]
\newtheorem{theo}[lemm]{Theorem}
\newtheorem{coro}[lemm]{Corollary}
\newtheorem{prop}[lemm]{Proposition}

\newtheorem*{claim*}{Claim}

\theoremstyle{definition}
\newtheorem{defi}[lemm]{Definition}
\newtheorem{rema}[lemm]{Remark}

\newtheorem{exam}[lemm]{Example}

\def\blank{\underline{\hphantom{A}}}

\makeatletter
\@namedef{subjclassname@2020}{%
  \textup{2020} Mathematics Subject Classification}
\makeatother

\def\citestacks#1{\cite[\href{https://stacks.math.columbia.edu/tag/#1}{Tag #1}]{stacks-project}}

\DeclareRobustCommand\longtwoheadrightarrow{\relbar\joinrel\twoheadrightarrow}
     
\newcommand{\set}[1]{\left\{ \, #1 \, \right\}}
\def\abs#1{\left\lvert#1\right\rvert}

\usepackage{xpatch}
\makeatletter   
\xpatchcmd{\@tocline}
{\hfil\hbox to\@pnumwidth{\@tocpagenum{#7}}\par}
{\ifnum#1<0\hfill\else\dotfill\fi\hbox to\@pnumwidth{\@tocpagenum{#7}}\par}
{}{}
\makeatother 

\makeatletter 
\def\l@subsection{\@tocline{2}{0pt}{3pc}{6pc}{}} 
\makeatother

\title[Stability conditions on products and Hilbert schemes]{Stability conditions on products of curves and Hilbert schemes of surfaces}

\begin{document}

\author[C.~Li]{Chunyi Li}
\address{\parbox{0.9\textwidth}{Mathematics Institute (WMI), University of Warwick\\[1pt]
Coventry, CV4 7AL, United Kingdom
\vspace{1mm}}}
\email{C.Li.25@warwick.ac.uk}
\urladdr{\url{https://sites.google.com/site/chunyili0401/}}

\author[E.~Macr\`i]{Emanuele Macr\`i}
\address{\parbox{0.9\textwidth}{Universit\'e Paris-Saclay, CNRS, Laboratoire de Math\'ematiques d'Orsay\\[1pt]
Rue Michel Magat, B\^at. 307, 91405 Orsay, France
\vspace{1mm}}}
\email{{emanuele.macri@universite-paris-saclay.fr}}
\urladdr{\url{https://www.imo.universite-paris-saclay.fr/~macri/}}

\author[A.~Perry]{Alexander Perry}
\address{\parbox{0.9\textwidth}{Department of Mathematics, University of Michigan\\[1pt]
530 Church Street, Ann Arbor, MI 48109, USA
\vspace{1mm}}}
\email{arper@umich.edu}
\urladdr{\url{http://www-personal.umich.edu/~arper/}}

\author[P.~Stellari]{Paolo Stellari}
\address{\parbox{0.9\textwidth}{Dipartimento di Matematica ``F.~Enriques'', Universit\`a degli Studi di Milano\\[1pt]
Via Cesare Saldini 50, 20133 Milano, Italy
\vspace{1mm}}}
\email{paolo.stellari@unimi.it}
\urladdr{\url{https://sites.unimi.it/stellari}}

\author[X.~Zhao]{Xiaolei Zhao}
\address{\parbox{0.9\textwidth}{Department of Mathematics, University of California Santa Barbara\\[1pt]
552 University Rd, Santa Barbara, CA 93117, USA
\vspace{1mm}}}
\email{xlzhao@ucsb.edu}
\urladdr{\url{https://sites.google.com/site/xiaoleizhaoswebsite/}}


\begin{abstract}
We prove that stability conditions on the derived category of a product of curves of positive genus are uniquely determined by their central charge and the phase of skyscraper sheaves.
As an application, we construct stability conditions on Hilbert schemes of points on certain surfaces, including some K3 surfaces of Kummer type.
\end{abstract}

\maketitle

\setcounter{tocdepth}{2}
\tableofcontents


\section{Introduction}\label{sec:intro}

The theory of Bridgeland stability conditions \cite{Bridgeland:Stab} has been highly influential in the past two decades, with applications to several areas in algebraic geometry, ranging from moduli spaces (see e.g.~\cite{AB+:MMP,BM:MMP,LZ:MMP,Fey:Mukai,BLMNPS:Family}), to counting invariants (see e.g.~\cite{AF+:Quantum,KS:Stability,FT:DT,LR:Castelnuovo, LiuZ:Castelnuovo}), to quadratic differentials (see e.g.~\cite{BS:Quadratic,HKK:Flat}), and beyond. 
Despite these advances, it is still a difficult problem to construct stability conditions on the derived category of a high-dimensional projective variety. 
In the important case of varieties with trivial canonical bundle, the state of affairs is as follows: 
stability conditions have been constructed on all abelian threefolds and a number of Calabi--Yau threefolds \cite{MP:Abelian,BMS:StabCY3s,Li:CY3,FK+:CY3}, while in higher dimensions stability conditions are only known to exist on products of the known examples in dimension at most $3$ with elliptic curves \cite{LiuY:Products}. 

In the presence of a morphism to an abelian variety, stability conditions should be constrained by the extra symmetries of the situation; 
this idea was explored in~\cite{FLZ:Albanese} to study the case of irregular surfaces and abelian varieties.
In this paper, we take the next step, by analyzing a relative version of~\cite{FLZ:Albanese}. 
This leads to our main result, which says that stability conditions on products of curves of positive genus are completely determined by linear data. 
As an application, by using (and slightly expanding) the construction in~\cite{LiuY:Products} of stability conditions on products of curves, we construct the first examples of stability conditions on Hilbert schemes of points on (special) K3 surfaces.
This result should be regarded as the first step in a larger project which will be completed in the forthcoming papers~\cite{MPS:Deformation,LMPSZ:Deformation}, where we 
deduce the existence of stability conditions on very general polarized hyperk\"ahler manifolds of $\mathrm{K3}^{[n]}$-type.
Other applications of our results are contained  in~\cite{ChengY:Kummer} (to hyperk\"ahler manifolds of Kummer type) and in~\cite{HaidenSung:EEE} (to the triple product of a generic elliptic curve).

\subsection{Main results} 
For a triangulated category $\cD$ and a homomorphism $v \colon \rK_0(\cD) \to \Lambda$ from its Grothendieck group to a finite rank free abelian group, 
a stability condition $\sigma = (Z, \cP)$ on $\cD$ with respect to $(\Lambda, v)$ consists of two pieces of data --- 
a central charge $Z \colon \Lambda \to \C$ and a slicing $\cP$ of $\cD$ --- subject to various compatibility conditions. 
Importantly, any such $\sigma$ gives rise to a notion of $\sigma$-(semi)stable objects $E \in \cD$, each of which has a well-defined phase $\phi_{\cP}(E) \in \R$. 
Moreover, by Bridgeland's fundamental deformation theorem, the collection of all $\sigma$ forms a complex manifold $\Stab_{(\Lambda, v)}(\cD)$. 

When $\cD = \Db(X)$ for a smooth projective variety, one often assumes that $v$ factors through the numerical Grothendieck group of $X$, since in practice any naturally occurring~$v$ has this property. 
In this situation, if $X$ has finite Albanese morphism, then in \cite{FLZ:Albanese} it is shown that for any  
$\sigma = (Z, \cP) \in \Stab_{(\Lambda,v)}(\Db(X))$, the skyscraper sheaves $k(x)$ of closed points $x \in X$ are all $\sigma$-stable of the same phase $\phi_{\cP}(k(x))$. In particular, for any closed point $x \in X$ we obtain a continuous map 
\begin{equation}
\label{Stab-map-Hom}
\Stab_{(\Lambda,v)}(\Db(X))\longrightarrow \Hom_\Z(\Lambda,\C)\times \R, \qquad \sigma=(Z,\cP)\longmapsto (Z,\phi_{\cP}(k(x))), 
\end{equation} 
which is independent of the choice of $x \in X$. 
A natural class of varieties with finite Albanese are products of curves of positive genus. Our main result shows that, surprisingly, for such varieties the above map is injective, so that stability conditions are completely determined by linear data. 

\begin{theo}\label{thm:Main}
Let $X = C_1 \times \cdots \times C_n$, where $C_1, \dots, C_n$ are smooth projective curves of positive genus over an algebraically closed field $k$. 
Let $\Lambda$ be a finite rank free abelian group and let $v \colon \rK_0(X) \to \Lambda$ be a homomorphism which factors through the numerical Grothendieck group of~$X$. 
Then the map $\Stab_{(\Lambda,v)}(\Db(X))\longrightarrow \Hom_\Z(\Lambda,\C)\times \R$ defined in~\eqref{Stab-map-Hom} 
is injective. 
\end{theo}

The proof of the theorem is based on some results of independent interest about restricting stability conditions. 
Namely, for a morphism to an abelian variety we show that stability conditions on the total space restrict to the subcategory of objects set-theoretically supported on any fiber 
(Proposition~\ref{prop:Restriction}), and when the morphism factors through a morphism to a positive genus curve, we show that stability conditions are in fact determined by their central charge and restrictions to fibers (Proposition~\ref{prop:RestInjective}). 

Theorem~\ref{thm:Main} raises several natural questions that we do not answer in this paper, the main one being how to characterize the image of the map~\eqref{Stab-map-Hom}. 
On the other hand, the theorem already gives strong control of stability conditions on $X$. 
For instance, as an immediate corollary, we can detect whether a group of autoequivalences of $\Db(X)$ preserves a stability condition from its action on the central charge and skyscraper sheaves. 
In particular, we obtain the following result for automorphisms of $X$.

\begin{coro}\label{cor:Invariance}
In the setting of Theorem~\ref{thm:Main}, assume that a group $G$ acts on $X$ and $\Lambda$ in such a way that the homomorphism $v$ is $G$-equivariant. 
Then $\sigma = (Z, \cP) \in \Stab_{(\Lambda,v)}(\Db(X))$ is $G$-invariant if and only if $Z$ is $G$-invariant.
\end{coro}

Quite recently, Theorem~\ref{thm:Main} has gained relevance as it stands as a key ingredient in \cite{Li25} where it is proven that $\Db(X)$ always carries stability conditions, for $X$ a smooth projective variety over the complex numbers. In the same circle of ideas, one application in this paper is the explicit construction of stability conditions on specific but relevant examples: Hilbert schemes of points on special surfaces. 
For an abelian surface $A$, recall that the Kummer K3 surface $\Kum(A)$ is the minimal resolution of singularities of the quotient of $A$ by the action of~$-1$.  

\begin{theo}
\label{theorem-Hilbn}
Let $S$ be a smooth projective surface over an algebraically closed field of characteristic $0$, which is of one of the following two types: 
\begin{enumerate}
    \item \label{Hilbn-C1C2}
    $S = C_1 \times C_2$, where $C_1$ and $C_2$ are smooth projective curves of positive genus; or 
    \item \label{Hilbn-Kum}
    $S = \Kum(E_1 \times E_2)$, where $E_1$ and $E_2$ are elliptic curves. 
\end{enumerate}
Then for all $n \geq 1$, stability conditions exist on $\Db(\Hilb^n(S))$. 
\end{theo}

The proof of the theorem relies on a strengthening of the support property (see Definition~\ref{def:StabilityCondition}) for the stability conditions on a product $C_1 \times \cdots \times C_n$ of positive genus curves constructed in \cite{LiuY:Products}. 
Namely, in Theorem~\ref{thm:LiuY+} we prove the support property holds 
with respect to a natural homomorphism $v_n \colon \rK_0(C_1 \times \cdots \times C_n) \to \Lambda_n$, where $\Lambda_n \coloneqq (\Z \oplus \Z)^{\otimes n}$ is a free abelian group of rank $2^n$. 
Besides its use in the proof of Theorem~\ref{theorem-Hilbn}, Theorem~\ref{thm:LiuY+} is another important ingredient in the construction of stability conditions on smooth projective varieties in \cite{Li25} mentioned above. 

In these terms, 
the 
stability conditions we construct in Theorem~\ref{theorem-Hilbn} are with respect to a pair $(\Lambda, v)$, where $\Lambda = \Lambda_{2n}^{\fS_n}$ is the invariant lattice for the action of the symmetric group $\fS_n$ on $\Lambda_{2n} = ((\Z \oplus \Z) \otimes (\Z \oplus \Z))^{\otimes n}$ permuting the $n$ factors (see Remark~\ref{rmk:InvariantLatticeHilb}). 

\subsection{Outline of the proofs} 
The proof of Theorem~\ref{thm:Main} builds on the approach of~\cite{FLZ:Albanese}. 
The starting observation is that, on an abelian variety $A$, the only objects in $\Db(A)$ which are invariant with respect to tensoring by line bundles in $\Pic^0(A)$ are objects supported on a finite set of closed points. 
In Section~\ref{sec:Support}, 
we apply this idea to the case of a 
morphism $a\colon X\to A$ from a projective variety to an abelian variety, and show that a stability condition on $\Db(X)$ induces a stability condition on the subcategory $\Db_{X_p}(X)$ of objects supported on the fiber $X_p$, for each closed point $p \in A$ (see Proposition~\ref{prop:Restriction}). 
Building on this, for a morphism $f \colon X \to C$ to a curve of positive genus, we show that stability conditions restrict in the same way to subcategories supported on fibers over closed points $p \in C$, and that any stability condition is uniquely determined by its central charge and these restrictions (see Proposition~\ref{prop:RestInjective}). 
This allows us to prove Theorem~\ref{thm:Main} in  Section~\ref{subsec:Proof} by reducing to an analogous assertion for stability conditions on the subcategory $\Db_{x}(X)$ of objects supported at a closed point $x \in X$. 
Corollary~\ref{cor:Invariance} is an immediate consequence. 

In Section~\ref{subsec:Hilbert} we prove Theorem~\ref{theorem-Hilbn} by reducing to the case of products of curves. 
This uses two additional techniques: 
the inducing machinery for stability conditions~\cite{Pol:Family, MMS:Inducing}, which allows for descending invariant stability conditions on a variety to a quotient stack, 
and the BKR correspondence for Hilbert schemes of points on surfaces~\cite{BKR:MukaiMcKay,Hai:Hilb}, which identifies the derived category of such a Hilbert scheme with that of the stacky $n$-fold symmetric product of the surface.  
The construction of stability conditions on products with curves from \cite{LiuY:Products} is reviewed in Section~\ref{subsec:Products}. 
In Section~\ref{subsec:StabilityProductCurves}, we prove a stronger support property for the resulting  stability conditions on products of positive genus curves, which is needed in order to apply the inducing machinery to the relevant group actions. 

In Section~\ref{sec:Bridgeland} we gather some preliminaries on stability conditions that are needed in the above proofs, including a useful criterion for the equality of two stability conditions and some theorems about the behavior of stability conditions under group actions.

\subsection{Conventions}
A variety over a field $k$ is an integral scheme which is separated and finite type over $k$. 
A curve is a $1$-dimensional variety. 
For a variety $X$, we denote by $\Db(X)$ its bounded derived category of coherent sheaves. 
If $Y\subset X$ is a closed subset, we denote by $\Db_Y(X)$ the full subcategory of $\Db(X)$ consisting of those objects with (set-theoretic) support on $Y$. 
All functors are derived.

\subsection{Acknowledgements}
The paper benefited from many useful discussions with the following
people, whom we gratefully acknowledge: 
Arend Bayer, Warren Cattani, Yiran Cheng, Lie Fu, Alexander Kuznetsov, Shengxuan Liu, Alekos Robotis, Ian Selvaggi, and Benjamin Sung.

During the preparation of this paper, C.L.~ was supported by the Royal Society University Research Fellowship URF\textbackslash R1\textbackslash 201129 ``Stability condition and application in algebraic geometry''; E.M.~ was partially supported by the ERC Synergy grant ERC-2020-SyG-854361-HyperK; A.P.~ was partially supported by the NSF grant DMS-2112747, the NSF FRG grant DMS-2052750, the NSF CAREER grant DMS-2143271, and a Sloan Research Fellowship; P.S.~ was partially supported by the ERC Consolidator grant ERC-2017-CoG-771507-StabCondEn, by the research project PRIN 2017 ``Moduli and Lie Theory'', and by the research project FARE 2018 HighCaSt (grant number R18YA3ESPJ); and X.Z.~ was partially supported by
the NSF grants DMS-2101789 and the NSF FRG grant DMS-2052665.
Part of this work was also completed while the authors were in residence at the Simons Laufer Mathematical Sciences Institute in Spring 2024 under the support of the NSF grant DMS-1928930. 


\section{Preliminaries on stability conditions}\label{sec:Bridgeland}

In this section, after briefly reviewing the basics of stability conditions in Section~\ref{section-stability-basics}, 
we discuss two important ingredients for our main results: 
a criterion for the equality of two stability conditions (Section~\ref{section-comparing-stability})
and group actions on stability conditions (Section~\ref{section-group-actions}).   

Throughout this section, $\cD$ denotes a triangulated category.

\subsection{Stability conditions} 
\label{section-stability-basics}
We closely follow~\cite{Bridgeland:Stab}, as reviewed in \cite[Section~12]{BLMNPS:Family}.

\begin{defi}\label{def:slicing}
A \emph{slicing} $\cP$ of $\cD$ is a collection of full additive subcategories $\cP(\phi) \subset \cD$, indexed by $\phi\in \mathbb{R}$, satisfying:
\begin{enumerate}
\item\label{enum:slicing1} For all $\phi\in\mathbb{R}$, $\cP(\phi+1)=\cP(\phi)[1]$. 
\item\label{enum:slicing2} If $\phi_1>\phi_2$ and $E_1\in\cP(\phi_1)$, $E_2\in\cP(\phi_2)$, then $\Hom_{\cD}(E_1,E_2)=0$. 
\item\label{enum:slicing3}
For every object $E \in \cD$, there exists a finite sequence of morphisms
\[ 0 = E_0 \xrightarrow{s_1} E_1 \xrightarrow{s_2} \cdots \xrightarrow{s_m} E_m = E \]
such that $F_i\coloneqq\mathrm{cone}(s_i)\in\cP(\phi_i)$ for a sequence of real numbers 
$\phi_1 > \phi_2 > \cdots > \phi_m$. 
\end{enumerate}
\end{defi}

The nonzero objects of $\cP(\phi)$ are called \emph{semistable} of \emph{phase} $\phi$; the simple objects in $\cP(\phi)$ are called \emph{stable}. 
The sequence of morphisms appearing in condition~\eqref{enum:slicing3} is unique and called the \emph{Harder--Narasimhan (HN) filtration} of $E$, and the objects $F_i$ are called the \emph{Harder--Narasimhan factors} of $E$. 
For a nonzero object $E \in \cD$, we denote by $\phi_\cP^+(E)\coloneqq \phi_1$ and $\phi_\cP^-(E)\coloneqq \phi_m$ the maximal and minimal phases of the Harder--Narasimhan factors; when $E$ is semistable, these numbers coincide and are denoted simply by $\phi_{\cP}(E)$.
A slicing gives a collection of bounded t-structures, parameterized by $\phi\in\R$.
More precisely, for an interval $I \subset \R$, we write 
\begin{equation*}
 \cP(I) \coloneqq \set{E\in\cD \st \phi_\cP^+(E), \phi_\cP^-(E) \in I} = \langle \cP(\phi) \rangle_{\phi \in I} \subset \cD, 
\end{equation*}
where $\langle \cP(\phi) \rangle_{\phi \in I} \subset \cD$ denotes the extension closure of the subcategories $\cP(\phi)$ for $\phi \in I$. 
Then, for each $\phi\in\R$, we get a bounded t-structure on $\cD$ with heart $\cP((\phi, \phi+1])$.

The weakest notion of a stability condition consists of a slicing and a central charge which are suitably compatible:

\begin{defi}
Let $v \colon \rK_0(\cD)\to \Lambda$ be a homomorphism from the Grothendieck group of $\cD$ to an abelian group $\Lambda$. 
A \emph{pre-stability condition} on $\cD$ with respect to $(\Lambda,v)$ is a pair $(Z, \cP)$ where $\cP$ is a slicing of $\cD$ and $Z\colon\Lambda\to\C$ is a group homomorphism, which satisfy the following condition: for all $\phi \in \R$ and $0 \neq E \in \cP(\phi)$, we have $Z(v(E)) \in\R_{>0}\cdot e^{\mathfrak{i}\pi\phi}$. 
\end{defi}

We often suppress $v$ from the notation, for instance writing $Z(E)$ for $Z(v(E))$. 

\begin{defi}
    Suppose $\cD$ is a $k$-linear triangulated category, where $k$ is a field. 
    We say that $\cD$ is \emph{proper} over $k$ if for every $E, F \in \cD$, the $k$-vector space $\bigoplus_{n \in \Z} \Hom(E, F[n])$ is finite-dimensional. 
    In this case, the \emph{Euler form} is the bilinear form $\chi \colon \rK_0(\cD) \times \rK_0(\cD) \to \Z$ induced by the Euler characteristic 
    \begin{equation*}
        \chi(E,F) = \sum_{n} (-1)^n \dim_k \Hom(E,F[n]) 
    \end{equation*} 
    for $E, F \in \cD$. 
    The \emph{numerical Grothendieck group} of $\Knum(\cD)$ is the quotient of $\rK_0(\cD)$ by the right kernel of this form, i.e.\ by the subgroup of elements $\beta \in \rK_0(\cD)$ for which $\chi(\alpha, \beta) = 0$ for all $\alpha \in \rK_0(\cD)$. 
    We denote by $v_{\mathrm{num}} \colon \rK_0(\cD) \to \Knum(\cD)$ the canonical homomorphism, and we say that a homomorphism $v \colon \rK_0(\cD) \to \Lambda$ to an abelian group is \emph{numerical} if it factors through $v_{\mathrm{num}}$. 
\end{defi}

\begin{rema}
    The numerical Grothendieck group $\Knum(\cD)$ is particularly well-behaved when $\cD$ is of geometric origin, in the sense that it admits an admissible embedding into the derived category of a smooth projective variety. First, in this case $\cD$ admits a Serre functor, which implies that the left and right kernel of the Euler form coincide, so that $\Knum(\cD)$ can equivalently be described as the quotient of $\rK_0(\cD)$ by the left kernel. 
    Second, $\Knum(\cD)$ is a finite rank free abelian group by \cite[Lemma~12.7]{BLMNPS:Family}. 
\end{rema}

\begin{defi}\label{def:StabilityCondition}
Let $v \colon \rK_0(\cD)\to \Lambda$ be a homomorphism from the Grothendieck group of $\cD$ to a finite rank free abelian group $\Lambda$. 
A \emph{stability condition} on $\cD$ with respect to $(\Lambda,v)$ is a pre-stability condition $\sigma=(Z,\cP)$ satisfying the following \emph{support property}:
there exists a quadratic form $Q$ on the vector space $\Lambda_\R\coloneqq\Lambda\otimes\R$ such that
\begin{itemize}
\item $Q$ is negative definite on the kernel of the $\R$-linear extension $Z_{\R}\colon\Lambda_\R\to\C$, and
\item for any $\sigma$-semistable object $E \in \cD$, we have $Q(v(E)) \geqslant 0$. 
\end{itemize}
\end{defi}

\begin{rema}
The support property implies that for any $C > 0$, there can only be finitely many classes $\bv\in\Lambda$ with $\abs{Z(\bv)} < C$ and $Q(\bv) \geqslant 0$. 
It follows that for a stability condition $\sigma = (Z, \cP)$, 
$\cP(\phi)$ is a finite length category for every $\phi \in \R$, as the image of objects in $\cP(\phi)$ in $\R_{>0}\cdot e^{i\pi\phi}$ is a discrete set.
In other words, the support property guarantees the existence of a \emph{Jordan--H\"older filtration} of a semistable object $E$, which has stable filtration quotients.
\end{rema}

\begin{rema}
The support property admits an alternative formulation which is sometimes convenient. 
Fix any norm $\|\blank\|$ on the finite dimensional vector space $\Lambda_\R$.
Then a pre-stability condition $\sigma=(Z,\cP)$ satisfies the support property if and only if there exists a constant $C_\sigma>0$ such that, for any $\sigma$-semistable object $E\in\cD$, we have
\begin{equation}\label{eq:SupportPropertyEquiv}
|Z(v(E))|\geq C_\sigma\|v(E)\|.    
\end{equation}
\end{rema}

Sometimes we need to compare stability conditions with respect to different data $(\Lambda,v)$.
The following two elementary observations are useful for this purpose. 

\begin{rema}\label{rmk:Surjectivity}
Let $\sigma=(Z,\cP)$ be a stability condition on $\cD$ with respect to $(\Lambda,v)$.
Let $\Lambda'$ denote the image of $v$ in $\Lambda$ and $\iota\colon\Lambda'\hookrightarrow\Lambda$ the inclusion. If we set 
\[
Z'\coloneqq Z\circ\iota\colon\Lambda'\longrightarrow\C
\]
and let $v'$ be the induced morphism $\rK_0(\cD)\to\Lambda'$, then $\sigma'=(Z',\cP)$ is a stability condition on $\cD$ with respect to $(\Lambda',v')$.
\end{rema}

\begin{rema}\label{rmk:CompareSupportProperty}
Let $\sigma=(Z,\cP)$ be a stability condition on $\cD$ with respect to $(\Lambda,v)$, where $v$ is surjective.
Let $\overline{\Lambda}$ denote the image of $Z$ in $\C$, seen as an abstract finite rank free abelian group, and denote by $p\colon\Lambda\twoheadrightarrow\overline{\Lambda}$ the surjection.
If we set
\[
\overline{v}\coloneqq p\circ v\colon \rK_0(\cD)\longrightarrow\overline{\Lambda}
\]
and let $\overline{Z}$ be the induced group homomorphism $\overline{\Lambda}\to\C$, then $\overline{\sigma}=(\overline{Z},\cP)$ is a stability condition on $\cD$ with respect to $(\overline{\Lambda},\overline{v})$.

Indeed, it is immediate to see that $\overline{\sigma}$ is a pre-stability condition with respect to $(\overline{\Lambda},\overline{v})$.
Hence, we only have to check the support property.
By fixing a splitting of the surjection $p$ as real vector spaces, we can choose a norm $\|\blank\|$ on $\Lambda_\R$ extending orthogonally a norm on $\overline{\Lambda}_\R$.
Hence, by using~\eqref{eq:SupportPropertyEquiv}, we have
\[
|\overline{Z}(\overline{v}(E))|=|Z(v(E))|\geq C_\sigma\|v(E)\|\geq C_\sigma \|p(v(E))\| = C_\sigma\|\overline{v}(E)\|,
\]
for any semistable object $E\in\cD$, which is what we wanted.
\end{rema}

The set $\Stab_{(\Lambda,v)}(\cD)$ of all stability conditions with respect to $(\Lambda,v)$ has a natural topology induced by the following generalized metric function: 
for stability conditions $\sigma_1=(Z_1, \cP_1)$ and $\sigma_2=(Z_2, \cP_2)$ with respect to $(\Lambda, v)$, their distance is defined as
\[
\mathrm{dist}(\sigma_1,\sigma_2)\coloneqq\sup_{\substack{ 0 \neq E\in\cD}}\left\{\abs{\phi_{\cP_1}^-(E)-\phi_{\cP_2}^-(E)},\abs{\phi_{\cP_1}^+(E)-\phi_{\cP_2}^+(E)}, \left\|Z_1-Z_2\right\|\right\}\in [0,+\infty],
\]
where $\|\blank\|$ denotes any fixed norm on the finite-dimensional vector space $\mathrm{Hom}_{\mathbb Z}(\Lambda, \mathbb C)$.

\begin{theo}[Bridgeland's Deformation Theorem]\label{thm:Bridgeland}
The space $\Stab_{(\Lambda,v)}(\cD)$ of stability conditions
is naturally a complex manifold of dimension $\rank(\Lambda)$ such that the map forgetting the slicing
\[
\mathcal{Z} \colon\Stab_{(\Lambda,v)}(\cD) \longrightarrow \mathrm{Hom}_{\mathbb Z}(\Lambda, \mathbb C), \qquad \sigma =(Z,\cP) \longmapsto Z 
\]
is a local biholomorphism at every point of $\Stab_{(\Lambda,v)}(\cD)$.
\end{theo}

Theorem~\ref{thm:Bridgeland} can be made explicit as well, describing the image of $\mathcal{Z}$ in terms of $Q$ (see~\cite[Theorem 7.1]{Bridgeland:Stab} and \cite[Theorem 1.2]{Bayer:ShortProof}), but we will not need this.

Stability conditions can equivalently be described in terms of t-structures and stability functions on their hearts, as we briefly review now. 

\begin{defi}\label{def:StabilityFunctionHeart}
Let $\cA$ be an abelian category and let $v\colon \rK_0(\cA)\to \Lambda$ be a homomorphism from the Grothendieck group of $\cA$ to an abelian group $\Lambda$. 
A \emph{stability function} $Z$ on $\cA$ with respect to $(\Lambda, v)$ is a group homomorphism $Z\colon \Lambda \to \C$ such that for all $0 \neq E \in\cA$, the complex number $Z(v(E))$ is contained in the semi-closed upper half plane
\[
\H\sqcup \R_{<0} \coloneqq \{z\in\C  \st \Im z > 0, \text{ or } \Im z = 0\text{ and }\Re z < 0\}.
\]
\end{defi}

For $0\neq E \in \cA$ we define its \emph{phase} by $\phi(E) \coloneqq \frac{1}{\pi} \arg Z(E) \in (0, 1]$. 
\index{phi@$\phi(E)$, phase of an object}
An object $0\neq E \in \cA$ is called \emph{$Z$-semistable} if for all subobjects $0 \neq A \hookrightarrow E$, we have $\phi(A) \leqslant \phi(E)$.

\begin{defi}\label{def:satisfies_HN}
A stability function $Z$ on an abelian category $\cA$ satisfies the \emph{Harder--Narasimhan property} if every object $E \in\cA$ admits a \emph{Harder--Narasimhan filtration}, i.e. a sequence 
\[0 = E_0\hookrightarrow E_1 \hookrightarrow 
E_2\hookrightarrow \ldots \hookrightarrow E_m = E\]
of subobjects in $\cA$ such that $E_i /E_{i-1}$ is $Z$-semistable for $i = 1,\ldots, m$ and 
\begin{equation*} 
\phi (E_1 /E_0 ) > \phi (E_2 /E_1 ) > \cdots > \phi (E_m /E_{m-1}).
\end{equation*}
\end{defi}

Note that if $\cA \subset \cD$ is the heart of a bounded t-structure, then $\rK_0(\cD) = \rK_0(\cA)$, so a homomorphism $v \colon \rK_0(\cD) \to \Lambda$ is the same as a homomorphism $v \colon \rK_0(\cA) \to \Lambda$. 
\begin{lemm}[{\cite[Proposition~5.3]{Bridgeland:Stab}}]\label{lem:Bridgeland-stabviaheart} 
Let $v \colon \rK_0(\cD) \to \Lambda$ be a homomorphism to an abelian group $\Lambda$. 
The following data are equivalent: 
\begin{itemize}
    \item A pre-stability condition on $\cD$ with respect to $(\Lambda, v)$. 
    \item A pair $(Z, \cA)$ where $\cA \subset \cD$ is the heart of a bounded t-structure and $Z$ is a stability function on $\cA$ with respect to $(\Lambda,v)$ which satisfies the Harder--Narasimhan property. 
\end{itemize}
\end{lemm}

Explicitly, the correspondence assigns to a pre-stability condition $\sigma = (Z, \cP)$ the pair $(Z, \cA)$, where $\cA$ is the heart $\cA=\cP((0,1])$; conversely, given $(Z, \cA)$, there is a unique extension of the phases of objects in $\cA$ to a slicing $\cP$ on $\cD$, determined by property~\eqref{enum:slicing1} in Definition~\ref{def:slicing}. 
When $\Lambda$ is free of finite rank, this correspondence also extends to stability conditions, since the support property may be checked for $\sigma$-semistable objects in the heart $\cA$. 
When we want to stress the abelian category $\cA$ instead of the slicing, we will use the notation $\sigma=(Z,\cA)$ for a (pre-)stability condition on $\cD$.

\subsection{Comparing stability conditions} 
\label{section-comparing-stability}
For later use, we prove here two useful results for comparing pre-stability conditions. 

\begin{lemm}\label{lem:DistOfHearts}
Let $\cP_1$ and $\cP_2$ be two slicings on $\cD$.
Assume that there exist $a_-,a_+\in\Z$, $a_-<a_+$, with the following property: for any nonzero object $E\in\cP_1((0,1])$, there exist objects
\[
E^-\in\cP_1((-\infty,2])\cap\cP_2((-\infty,1-a_-])
\quad 
\text{and} 
\quad 
E^+\in\cP_1((-1,\infty))\cap\cP_2((1-a_+,\infty))
\]
such that
\[
\Hom_\cD(E,E^-)\neq0\quad \text{and} \quad\Hom_\cD(E^+,E)\neq0.
\]
Then
\[
\cP_2((0,1])\subset\cP_1((a_-,a_+]).
\]
\end{lemm}

\begin{proof}
Let $F\in\cP_2((0,1])$ be a nonzero object.
If we consider the bounded t-structure $\tau_1$ associated to the slicing $\cP_1$, we get a distinguished triangle
\begin{equation}\label{eq:DistOfHearts}
E\longrightarrow F \longrightarrow E',    
\end{equation}
where $0\neq E\in\cP_1((n-1,n])$ is the (shifted) minimal cohomology object of $F$ with respect to $\tau_1$ and  $E'\in\cP_1((-\infty,n-1])$, for some $n\in\Z$.
By assumption, there exists
\[
E^+\in\cP_1((n-2,\infty))\cap\cP_2((n-a_+,\infty))
\]
such that $\Hom_\cD(E^+,E)\neq0$.
By applying $\Hom_\cD(E^+,-)$ to~\eqref{eq:DistOfHearts}, we get an exact sequence
\[
\dots\rightarrow \Hom_\cD(E^+,E'[-1])\rightarrow\underbrace{\Hom_\cD(E^+,E)}_{\neq0}\rightarrow\Hom_\cD(E^+,F)\rightarrow\dots
\]
Since $E^+\in\cP_1((n-2,\infty))$ and $E'[-1]\in\cP_1((-\infty,n-2])$, we get
\[
\Hom_\cD(E^+,E'[-1])=0.
\]
Hence, $\Hom_\cD(E^+,F)\neq0$.
Since $E^+\in\cP_2((n-a_+,\infty))$ and $F\in\cP_2((0,1])$, we must have $n-a_+<1$, namely $n\leq a_+$.

Similarly, we get a distinguished triangle
\[
E'\longrightarrow F \longrightarrow E\longrightarrow E'[1],
\]
where $E\in\cP_1((n,n+1])$ and $E'\in\cP_1((n+1,+\infty))$ for a suitable $n \in \Z$.
Then we can argue as above using the object 
\[
E^-\in \cP_1((-\infty,n+2])\cap\cP_2((-\infty,n+1-a_-])
\]
to deduce that $n \geq a_-$, as we wanted.
\end{proof}

Later we will use the above result for $a_-=-1$ and $a_+=2$ in combination with the following lemma (which is similar to~\cite[Lemma 8.11]{BMS:StabCY3s}).

\begin{lemm}\label{lem:SameCentralCharge}
Let $\sigma_1=(Z_1,\cP_1)$ and $\sigma_2=(Z_2,\cP_2)$ be pre-stability conditions on $\cD$ with respect to a fixed $(\Lambda,v)$.
Assume that $Z_1=Z_2$ and that
\[
\cP_2((0,1])\subset\cP_1((-1,2]).
\]
Then $\sigma_1=\sigma_2$.
\end{lemm}

\begin{proof}
We only need to show that $\cP_1((0,1])\subset\cP_2((0,1])$; indeed, then the heart of the bounded t-structure corresponding to $\cP_1$ must be contained in, and hence coincide with, that of $\cP_2$. 
Let $E\in\cP_1((0,1])$ be a nonzero object. 
If we consider the bounded t-structure associated to $\cP_2$, then the truncation triangle at the minimal cohomological degree of $E$ has the form 
\begin{equation}\label{eq:SameCentralCharge}
F \longrightarrow E \longrightarrow F',
\end{equation}
where $0 \neq F\in\cP_2((n,n+1])$ and $F'\in\cP_2((-\infty,n])$ for some $n \in \Z$.
We claim that $n\leq0$, and hence that $E \in \cP_2((-\infty,1])$. 
Indeed, if $n>0$, we have 
\[
E\in\cP_1((0,1])\subset\cP_1((-\infty,n]).
\]
By assumption, we also have
\[
F'[-1]\in\cP_2((-\infty,n-1])\subset\cP_1((-\infty,n]).  
\]
Hence from the triangle~\eqref{eq:SameCentralCharge}, we deduce that $F\in\cP_1((-\infty,n])$ as well.
Hence, using our assumption again, we have
\[
F\in\cP_2((n,n+1])\cap\cP_1((-\infty,n])\subset\cP_1((n-1,n]). 
\]
Thus $Z_1(F)$ and $Z_2(F)$ live in different semi-closed half planes, which contradicts the assumption $Z_1 = Z_2$. 

Similarly, we get a distinguished triangle
\[
F'\longrightarrow E \longrightarrow F,  
\]
where $0 \neq F\in\cP_2((n,n+1])$ and $F'\in\cP_2((n+1,+\infty))$.
As above, we deduce that $n\geq0$, and hence conclude that $E\in\cP_2((0,\infty))$, as we wanted.
\end{proof}

\subsection{Group actions} 
\label{section-group-actions}

The group of autoequivalences of $\cD$ acts on $\Stab_{(\Lambda,v)}(\cD)$.
More precisely, 
let $G\subset\Aut(\cD)$ be a subgroup of the group of (isomorphism classes of) autoequivalences of $\cD$, 
let $\Lambda$ be a finite rank free abelian group equipped with a $G$-action, and let $v \colon \rK_0(\cD) \to \Lambda$ be a $G$-equivariant homomorphism. 
Then $G$ acts on $\Stab_{(\Lambda,v)}(\cD)$. 
Explicitly, $g \in G$ acts on $\sigma=(Z,\cP)\in\Stab_{(\Lambda,v)}(\cD)$ by $g \cdot \sigma = (g \cdot Z, g \cdot \cP)$, where 
\begin{equation*}
    g \cdot Z \coloneqq Z \circ g^{-1} \quad \text{and} \quad 
    (g \cdot \cP)(\phi) \coloneqq g(\cP(\phi)) \text{ for all } \phi \in \R,  
\end{equation*}
and we write simply $g$ for the automorphism of $\Lambda$ and autoequivalence of $\cD$ given by the group actions. 
We say that $\sigma$ is \emph{$G$-invariant} if it is fixed by this $G$-action.

Later we will need the following theorem to pass from products of curves to our results about Hilbert schemes of points. 
\begin{theo}[{\cite[Proposition 2.2.3]{Pol:Family}, \cite[Theorem 1.1]{MMS:Inducing}, \cite[Theorem 10.1]{BMS:StabCY3s}}]
\label{theorem-G-invariant-stability}
     Let $X$ be a variety over a field $k$ with an action by a finite group $G$ of order prime to $\mathrm{char}(k)$. 
     Let $\Lambda$ be a finite rank free abelian group equipped with a $G$-action, and let $v \colon \rK_0(X) \to \Lambda$ be a $G$-equivariant homomorphism.
     Let $\pi \colon X \to [X/G]$ be the canonical morphism to the quotient stack. 
     Given any $G$-invariant stability condition $\sigma \in \Stab_{(\Lambda, v)}(\Db(X))$, we define: 
     \begin{enumerate}
     \item $v^G \colon \rK_0([X/G]) \to \Lambda^G$ to be the homomorphism through which the composition 
     \[
\rK_0([X/G])\xlongrightarrow{\pi^*}\rK_0(X)\xlongrightarrow{v}\Lambda
\]
factors; 
    \item $\cP^G(\phi) \coloneqq \set{ E \in \Db([X/G]) \st \pi^*E \in \cP(\phi)}$ for any $\phi \in \R$; and 
    \item $Z^G \colon \Lambda^G \to \C$ as the restriction of $Z$ to $\Lambda^G \subset \Lambda$. 
     \end{enumerate}
    Then $\sigma^G \coloneqq (Z^G, \cP^G)$ is a stability condition on $\Db([X/G])$ with respect to $(\Lambda^G, v^G)$. 
\end{theo}

\begin{rema}
    There is also a version of Theorem~\ref{theorem-G-invariant-stability} for suitable triangulated categories $\cD$ in place of $\Db(X)$; see \cite[Theorem 4.8]{enriques-category-stability} and \cite[Theorem 2.26]{dell-stability-quotients}. 
    In this setting, the role of $\Db([X/G])$ is played by the $G$-invariant category $\cD^{G}$. 
\end{rema}

The above discussion concerned the action of a discrete group on $\cD$. 
We will also need to consider the action of an algebraic group in the following sense. 

\begin{defi}
\label{definition-algebraic-G-act-D}
    Let $X$ be a variety over an algebraically closed field $k$. 
    An action of an algebraic group $G$ over $k$ on $\Db(X)$ consists of 
    \begin{itemize} 
    \item a group homomorphism $G(k) \to \Aut(\Db(X))$, $g \mapsto \Phi_g$, and 
    \item a kernel $K \in \Db(G \times X \times X)$, 
    \end{itemize} 
    which satisfy the following conditions: 
    \begin{enumerate}
    \item The Fourier--Mukai functor $\Phi_{K}$ given by $E \mapsto \mathrm{pr}_{13*}(\mathrm{pr}_{2}^*(E) \otimes K)$ sends $\Db(X)$ to $\Db(G \times X)$, and hence can be regarded as a functor 
    \begin{equation*}
        \Phi_K \colon \Db(X) \to \Db(G \times X). 
    \end{equation*}
    \item For every $g \in G(k)$, the Fourier--Mukai functor $\Phi_{K_g}$ associated to $K_g \coloneqq K\vert_{g \times X \times X}$ sends $\Db(X)$ to $\Db(X)$\footnote{This hypothesis follows by base change from the one on $\Phi_K$, as long as $G$ is smooth.}, and there is an isomorphism of functors $\Phi_g \cong \Phi_{K_g}$. 
    \end{enumerate}
    
    If $Y \subset X$ is a closed subvariety, then an action of $G$ on $\Db_Y(X)$ is defined to be an action of $G$ on $\Db(X)$ such that for every $g \in G(k)$, the autoequivalence $\Phi_g \in \Aut(\Db(X))$ preserves the subcategory $\Db_Y(X) \subset \Db(X)$. 
\end{defi}

\begin{rema}
\label{remark-PhiT}
    In the setting of Definition~\ref{definition-algebraic-G-act-D}, for any subvariety $T \subset G$ we set 
\begin{equation*}
    K_T \coloneqq K|_{T \times X \times X} \in \Dqc(T \times X \times X), 
\end{equation*}
where $\Dqc(T \times X \times X)$ denotes the unbounded derived category of quasi-coherent sheaves, and denote by 
\begin{equation*}
    \Phi_T \colon \Dqc(X) \to \Dqc(T \times X), \quad 
    E \mapsto \mathrm{pr}_{13*}(\mathrm{pr}_2^*(E) \otimes K_T) 
\end{equation*}
the associated Fourier--Mukai functor. 
This construction is compatible with base change, in the sense that for any $T' \subset T$, we have $(\Phi_T)|_{T' \times X} \simeq \Phi_{T'}$, where 
\begin{equation*}
    (-)|_{T' \times X} \colon \Dqc(T \times X) \to \Dqc(T' \times X)
\end{equation*} 
denotes the restriction functor. 
Moreover, if $T$ is either a point or an open subset of $G$, then $\Phi_T$ in fact restricts to a functor $\Db(X) \to \Db(T \times X)$ between bounded derived categories; indeed, when $T$ is a point or all of $G$, this follows from the definition of a $G$-action on $\Db(X)$, and the case where $T$ is open follows from the case where $T = G$ and the identification $(\Phi_G)|_{T \times X} \simeq \Phi_T$. 
\end{rema}

\begin{rema}
There is also a notion of an action of an algebraic group $G$ on a suitably enhanced smooth and proper $k$-linear triangulated category $\cD$. 
Namely, under a mild connectivity assumption, there is a group algebraic space $\Aut(\cD/k)$ classifying the autoequivalences of $\cD$ \cite[Proposition 8.2]{PI-abelian-3fold}, and an action of $G$ on $\cD$ can then be defined as a homomorphism $G \to \Aut(\cD/k)$ of group algebraic spaces. 
This definition is arguably more conceptual than Definition~\ref{definition-algebraic-G-act-D}, but we have eschewed it for the sake of simplicity. 
\end{rema}

\begin{exam}
    Let $X$ be a proper variety equipped with an action by an algebraic group $G$. 
    Then $G$ acts on $\Db(X)$. 
    Namely, we have a homomorphism $G(k) \to \Aut(\Db(X))$ given by $g \mapsto g_*$, and if $\Gamma \subset G \times X \times X$ is the graph of the action morphism $G \times X \to X$, then $K = \cO_{\Gamma}$ is a kernel satisfying the required compatibility condition. 
\end{exam}

\begin{exam}
\label{example-Pic0-action}
Let $f \colon X \to B$ be a morphism from 
a proper variety $X$ 
to a smooth projective variety $B$ over an algebraically closed field $k$. 
Then the Picard variety $\Pic^0(B)$ acts on $\Db(X)$ in the sense of Definition~\ref{definition-algebraic-G-act-D}. 
Explicitly, given a line bundle $L$ corresponding to a $k$-point of $\Pic^0(B)$, its action on $\Db(X)$ is given by the autoequivalence $- \otimes f^*L$. 
Further, if $P$ denotes the Poincar\'{e} line bundle on $\Pic^0(B) \times B$ 
and $\Delta \colon X \to X \times X$ the diagonal, 
then it is straightforward to check that 
\begin{equation*}
    K = (\id_{\Pic^0(B)} \times \Delta)_* (\id_{\Pic^0(B)} \times f)^* P \in \Db(\Pic^0(B) \times X \times X)
\end{equation*}
is a kernel which satisfies the required compatibility condition. 

Note that the above action preserves the support of objects in $\Db(X)$. Hence for any closed subvariety $Y \subset X$, we also obtain an action of $\Pic^0(B)$ on $\Db_Y(X)$.
\end{exam}

Earlier, given an action of a discrete group $G$ on $\cD$, we explained that $G$ acts on $\Stab_{(\Lambda, v)}(\cD)$ when $v \colon \rK_0(\cD) \to \Lambda$ is a $G$-equivariant homomorphism. 
If $X$ is a variety over $k = \bar{k}$ with a closed subvariety $Y \subset X$ and $G$ is an algebraic group acting on $\cD = \Db_Y(X)$, we can apply this definition at the level of $k$-points; in particular, if $\rK_0(\Db_Y(X)) \to \Lambda$ is a $G(k)$-equivariant homomorphism, we say that $\sigma \in \Stab_{(\Lambda, v)}(\Db_Y(X))$ is $G$-invariant if it is $G(k)$-invariant. 
Similarly, if $\tau = (\Db_Y(X)^{\leq 0}, \Db_Y(X)^{\geq 0})$ is a t-structure on $\Db_Y(X)$, we say $\tau$ is $G$-invariant if for every $g \in G(k)$ the t-structure $g \cdot \tau =  (\Phi_g(\Db_Y(X)^{\leq 0}), \Phi_g(\Db_Y(X)^{\geq 0}))$ coincides with $\tau$. 

\begin{theo}[{\cite[Theorem 3.5.1]{Pol:Family}}]
\label{theorem-invariant-stability}
    Let $X$ be a 
    variety over an algebraically closed field $k$, and let $Y \subset X$ be a closed subvariety. 
    Assume we are given an action of a connected algebraic group $G$ on $\Db_Y(X)$. 
    \begin{enumerate}
        \item \label{tau-invariant-connected-group} If $\tau$ is a noetherian\footnote{Meaning that its heart is a noetherian abelian category.} bounded t-structure on $\Db_Y(X)$, then $\tau$ is $G$-invariant. 
        \item \label{sigma-invariant-connected-group} 
        If $v \colon \rK_0(\Db_Y(X)) \to \Lambda$ is a homomorphism to a finite rank free abelian group which is $G(k)$-equivariant when $\Lambda$ is equipped with the trivial $G(k)$-action, 
        then every stability condition $\sigma \in \Stab_{(\Lambda, v)}(\Db_Y(X))$ is $G$-invariant. 
    \end{enumerate}
\end{theo}

\begin{rema}
\label{remark-numerical-v-G-equivariant}
    In the setup of the above theorem, it is easy to see that $G(k)$ acts trivially on the numerical Grothendieck group $\Knum(\Db_Y(X))$ (when it is defined, e.g. if $X$ is smooth and proper), since the Euler characteristic is constant in families. 
    In particular, the assumption of $G(k)$-equivariance on the homomorphism $v$ in~\eqref{sigma-invariant-connected-group} is automatically satisfied if $v$ is numerical. 
\end{rema}

\begin{proof} 
\eqref{tau-invariant-connected-group} When $Y = X$, so that $\Db_Y(X) = \Db(X)$, this is \cite[Theorem 3.5.1]{Pol:Family}, but there is a gap in the proof. So we first explain a complete proof in this case due to Arend Bayer, and then indicate the necessary modifications for general $Y \subset X$. 

Let $\tau = (\Db(X)^{\leq 0}, \Db(X)^{\geq 0})$ be a noetherian bounded t-structure on $\Db(X)$ with heart~$\cA$. 
The key ingredient in the proof is the base changed t-structure $\tau_S$ on $\Db(S \times X)$ constructed in \cite[Theorem~3.3.6]{Pol:Family} for any finite type $k$-scheme $S$, whose heart we denote by $\cA_S$, with the following properties: 
\begin{enumerate}
    \item \label{tauS-1} $\tau_S$ is a noetherian bounded t-structure on $\Db(S \times X)$, such that for $S = \Spec(k)$ we have $\tau_S = \tau$. 
    \item $\tau_S$ is local over $S$ in the sense that for any open subset $U \subset S$ the restriction functor $\Db(S \times X) \to \Db(U \times X)$ is t-exact with respect to $\tau_{S}$ and $\tau_{U}$. 
    \item \label{tauS-3} $\tau_S$ satisfies openness of the heart  \cite[Proposition~2.3.7]{Pol:Family} in the sense that for any object $E \in \Db(S \times X)$ and point $s \in S(k)$ such that $E|_{s \times X} \in \cA_s = \cA$, there exists an open neighborhood $s \in U \subset S$ such that $E|_{U \times X} \in \cA_U$. 
\end{enumerate}

Now we prove that $\tau$ is $G$-invariant. 
Let $g \in G(k)$ and $E \in \cA$. We must show that $\Phi_g(E) \in \cA$. 
For sake of contradiction, assume that $\Phi_g(E) \notin \Db(X)^{\geq 0}$. 
Then there exists an object $F \in \cA$ and an integer $n > 0$ such that there is a nonzero morphism $f \colon F[n] \to \Phi_g(E)$. 
For any nonempty subvariety $T \subset G$, we obtain a  morphism $\Phi_T(f) \colon \Phi_T(F)[n] \to \Phi_T(\Phi_g(E))$ which is nonzero, where $\Phi_T$ is defined as in Remark~\ref{remark-PhiT}.
Indeed, if $h \in T(k)$ is a point, then the restriction of $\Phi_T(f)$ to $h \times X$ is the morphism obtained by applying the autoequivalence $\Phi_h$ to the nonzero morphism $f$. 

On the other hand, we have 
\begin{equation*}
\Phi_G(\Phi_g(E))\vert_{g^{-1} \times X} \cong 
\Phi_{g^{-1}} \Phi_g(E) \cong E \in \cA, 
\end{equation*}
and thus by openness of the heart for the t-structure $\tau_G$, we obtain an open neighborhood $g^{-1} \in U \subset G$ such that $\Phi_U(\Phi_g(E)) \in \cA_{U}$. 
Similarly, if $e \in G(k)$ denotes the identity, then we find a neighborhood $e \in V \subset G$ such that $\Phi_V(F) \in \cA_V$. 
As $G$ is connected, the intersection $W = U \cap V$ is nonempty, and by locality of the t-structure $\tau_G$, upon restriction to this open subset the objects  
$\Phi_W(\Phi_g(E)) = \Phi_U(\Phi_g(E))|_{W \times X}$ and $\Phi_W(F) = \Phi_V(F)|_{W \times X}$ are contained in the heart $\cA_W$. 
This implies that any morphism $\Phi_W(F)[n] \to \Phi_W(\Phi_g(E))$ must vanish, contradicting the conclusion of the previous paragraph. 

This proves that $\Phi_g(E) \in \Db(X)^{\geq 0}$. Arguing similarly, we find that $\Phi_g(E) \in \Db(X)^{\leq 0}$, and hence $\Phi_g(E) \in \cA$, as required. 

Finally, we consider the case when $Y \subset X$ is an arbitrary closed subvariety. 
Let $\tau$ be a noetherian bounded t-structure on $\Db_Y(X)$. 
Then the arguments in \cite{Pol:Family} easily extend to provide a base changed t-structure $\tau_S$ on $\Db_{S \times Y}(S \times X)$ for any finite type $k$-scheme $S$, satisfying the analogues of the properties~\eqref{tauS-1}-\eqref{tauS-3} above. 
With this in hand, the argument above works verbatim for $\Db_Y(X)$. 
This completes the proof of part~\eqref{tau-invariant-connected-group} of the theorem. 

\eqref{sigma-invariant-connected-group} 
If $\sigma = (Z, \cA) \in \Stab_{(\Lambda, v)}(\Db_Y(X))$ is a stability condition such that $Z$ is contained in $\Hom_{\Z}(\Lambda, \Q \oplus i \Q) \subset \Hom_{\Z}(\Lambda, \C)$, then as is well-known (see \cite[Lemma 12.12]{BLMNPS:Family}), the heart $\cA$ is noetherian. 
Therefore, by~\eqref{tau-invariant-connected-group} the heart $\cA$ is $G(k)$-invariant. 
Since by assumption $G(k)$ acts trivially on $\Lambda$, 
the stability condition $\sigma$ is therefore $G(k)$-invariant. 
Moreover, it follows from Theorem~\ref{thm:Bridgeland} that the collection of $\sigma = (Z, \cA)$ such that $Z$ takes values in $\Q \oplus i\Q$ is dense in $\Stab_{(\Lambda, v)}(\cD)$. 
Thus, $G(k)$ acts trivially on a dense subset of $\Stab_{(\Lambda, v)}(\cD)$, and hence on all of $\Stab_{(\Lambda, v)}(\cD)$ as the action is continuous. 
\end{proof}

\begin{rema}
For the proofs of our main results, we in fact only need Theorem~\ref{theorem-invariant-stability} in the special case where $Y = X$, so that $\Db_Y(X) = \Db(X)$. 
More precisely, as the reader may verify by tracing through our proofs, we only need to apply the theorem for stability conditions $\sigma \in \Stab_{(\Lambda, v)}(\Db_Y(X))$ which are restricted from a stability condition $\sigma'$ (satisfying the same hypotheses) on $\Db(X)$. 
However, our intermediate results Proposition~\ref{prop:Restriction} and Proposition~\ref{prop:RestInjective} proved below are more naturally stated for stability conditions that are not restricted from $\Db(X)$, and in this form they depend on Theorem~\ref{theorem-invariant-stability} for general $Y \subset X$. 
\end{rema}


\section{Stability conditions on categories with support and products of curves} 
\label{sec:Support}

The goal of this section is to prove our first main result, Theorem~\ref{thm:Main}. 
Our proof is based on some results of independent interest about stability conditions on derived categories with support on a prescribed closed subset. Namely, in the presence of a morphism to an abelian variety, we prove that such stability conditions can be restricted to fibers (Proposition~\ref{prop:Restriction}), and when the morphism factors through a morphism to a positive genus curve, we prove that stability conditions are moreover determined by their central charge and restrictions to fibers (Proposition~\ref{prop:RestInjective}).
Using this, we prove Theorem~\ref{thm:Main} in Section~\ref{subsec:Proof}, by reducing to an easy analogous assertion for stability conditions on the subcategory of objects supported at a point. 

Throughout this section, we work over an algebraically closed field $k$. 

\subsection{Restricting stability conditions to fibers over abelian varieties}
\label{subsec:Restricting}

If $X$ is a projective variety equipped with a morphism to an abelian variety $A$, then by  Example~\ref{example-Pic0-action} the Picard variety $\Pic^0(A)$ acts on $\Db(X)$, or more generally on $\Db_Y(X)$ for any closed subset $Y \subset X$.
Our goal is to show that in this setting, stability conditions can be restricted to fibers, or more precisely to the subcategory of objects set-theoretically supported on a fiber: 

\begin{prop} \label{prop:Restriction}
Let $a \colon X \to A$ be a morphism from a projective variety to an abelian variety over an algebraically closed field $k$, and let $Y \subset X$ be a closed subset. 
Let $v \colon \rK_0(\Db_Y(X)) \to \Lambda$ be a homomorphism to a finite rank free abelian group such that $v$ is $\Pic^0(A)(k)$-equivariant when $\Lambda$ is equipped with the trivial action. 
Let $\sigma = (Z, \cP) \in \Stab_{(\Lambda, v)}(\Db_Y(X))$ be a stability condition. 
For any closed point $p \in A$, let $Y_p \subset Y$ be the fiber of $Y$ over $p$ and define: 
\begin{enumerate}
    \item $v_p \colon \rK_0(\Db_{Y_p}(X)) \to \rK_0(\Db_Y(X)) \xrightarrow{v} \Lambda$ as the natural composition; and  
    \item $\cP_p(\phi) \coloneqq \cP(\phi) \cap \Db_{Y_p}(X)$ for any $\phi \in \R$. 
\end{enumerate}
Then $\sigma_p \coloneqq (Z, \cP_p)$ is a stability condition on $\Db_{Y_p}(X)$ with respect to $(\Lambda, v_p)$, and 
the map
\begin{equation*}
\Stab_{(\Lambda,v)}(\Db_Y(X))\longrightarrow\Stab_{(\Lambda,v_p)}(\Db_{Y_p}(X)), \quad \sigma=(Z,\cP)\mapsto\sigma_p=(Z,\cP_p), 
\end{equation*} 
is holomorphic.
\end{prop}

\begin{rema}
    By Remark~\ref{remark-numerical-v-G-equivariant}, the condition that $v$ is $\Pic^0(A)(k)$-equivariant is automatically satisfied if $v$ is numerical. 
\end{rema}

We begin by establishing two preliminary lemmas. 
The first is a relative version of \cite[Lemma~2.12]{FLZ:Albanese}. 
For an object $E \in \Db(X)$, 
we denote by $\Supp(E)$ its support, i.e. the union of the supports of its cohomology sheaves, regarded as a closed subset of $X$ (or equivalently as a subscheme with the reduced induced structure). 

\begin{lemm}\label{lem:FinSupport}
Let $a \colon X \to A$ be a morphism from a projective variety to an abelian variety. 
Let $E\in\Db(X)$ be an object such that $E\otimes a^*L\cong E$ for all $L \in\Pic^0(A)$.
Then $a(\Supp(E)) \subset A$ consists of a finite set of points. 
\end{lemm}

\begin{proof}
The condition $E \otimes a^*L \cong E$ is inherited by the cohomology sheaves of $E$, so we may reduce to the case where $E$ is a coherent sheaf. 
In this case, we observe that by Serre vanishing~\citestacks{0D39} we have
\[
a(\Supp(E))=\Supp(a_*(E\otimes M^{\otimes n}))
\]
for all $n\gg0$, where $M$ is an ample line bundle on $X$.

On the other hand, by the projection formula and our assumption on $E$, we have
\[
a_*(E\otimes M^{\otimes n})\otimes L \cong a_*(E\otimes M^{\otimes n}) 
\]
for all $L \in\Pic^0(A)$.
We can thus apply \cite[Proposition 11.8]{Pol:Abelian} to deduce that $a_*(E\otimes M^{\otimes n})$ is supported on finitely many points, as we wanted.
\end{proof}

\begin{lemm}\label{lem:HNfactors}
In the setup of Proposition~\ref{prop:Restriction}, 
the Harder--Narasimhan factors with respect to~$\sigma$ of any object $E \in \Db_{Y_p}(X)$ are contained in $\Db_{Y_p}(X)$.  
\end{lemm}

\begin{proof}
Note that because $E$ is supported over $p \in A$, we have $E \otimes a^*L \cong E$ for all $L \in \Pic^0(A)$. 
Since by Theorem~\ref{theorem-invariant-stability}\eqref{sigma-invariant-connected-group} the stability condition $\sigma$ is invariant under the action of $\Pic^0(A)$, 
we conclude that the Harder--Narasimhan filtration of 
$E$ with respect to $\sigma$ 
is preserved by the action of $\Pic^0(A)$, and hence the Harder--Narasimhan factors $F_1, \dots, F_m$ of $E$ also satisfy $F_i \otimes a^*L \cong F_i$ for all $i$ and $L \in \Pic^0(A)$.  
By Lemma~\ref{lem:FinSupport}, we conclude that $a(\Supp(F_i))$ is a finite set of points for each $i$. 

We claim that in fact $a(\Supp(F_i)) = \set{p}$ for every $i$, or in other words $F_i \in \Db_{Y_p}(X)$. 
By induction on the length $m$ of the Harder--Narasimhan filtration of $E$, it suffices to show that the factor $F_1$ is contained in $\Db_{Y_p}(X)$. 
Since $\Gamma = a(\Supp(F_1))$ is a finite set of points, $F_1$ decomposes as a direct sum $F_1 = \bigoplus_{q \in \Gamma} F_{1, q}$, where $F_{1,q}$ is set-theoretically supported on $Y_q$. 
Thus, it suffices to show that if $G \in \cP(\phi)$, where $\phi = \phi_{\cP}(F_1)$, is supported on $Y_q$ for a closed point $q \in A \setminus \set{p}$, then $\Hom(G,F_1) = 0$. 
We have a distinguished triangle 
\begin{equation*}
    F_1 \to E \to E'
\end{equation*}
where $E' \in \cP((-\infty, \phi))$, which gives an exact sequence 
\begin{equation*}
    \cdots \to \Hom(G, E'[-1]) \to \Hom(G, F_1) \to \Hom(G, E) \to \cdots. 
\end{equation*}
The first term vanishes since $E'[-1] \in \cP(-\infty, \phi-1)$, and the third term vanishes since $G$ and $E$ are supported on different fibers of the morphism $a \colon X \to A$. 
\end{proof}

\begin{proof}[Proof of Proposition~\ref{prop:Restriction}]
First we claim that $\cP_p$ is a slicing of $\Db_{Y_p}(X)$. 
Indeed, conditions~\eqref{enum:slicing1} and~\eqref{enum:slicing2} of Definition~\ref{def:slicing} are immediate, while condition~\eqref{enum:slicing3} follows from Lemma~\ref{lem:HNfactors}. 
The fact that $\sigma_p$ is a pre-stability condition with respect to $(\Lambda, v_p)$ which satisfies the support property then follows directly from the fact that $\sigma$ is a stability condition with respect to $(\Lambda, v)$. Finally, the holomorphicity of the map $\Stab_{(\Lambda,v)}(\Db_Y(X)) \to \Stab_{(\Lambda,v_p)}(\Db_{Y_p}(X))$ follows from the observation that, by construction, we have 
\begin{equation*}
\mathrm{dist}(\sigma_{1,p},\sigma_{2,p})\leq \mathrm{dist}(\sigma_1,\sigma_2)
\end{equation*}
for all $\sigma_1,\sigma_2\in\Stab_{(\Lambda,v)}(\Db_Y(X))$.
\end{proof}

\subsection{Restricting stability conditions to fibers over curves}
\label{subsec:Curves}

If $X$ is a projective variety equipped with a morphism to a smooth projective curve $C$, then by  Example~\ref{example-Pic0-action} the Picard variety $\Pic^0(C)$ acts on $\Db(X)$, or more generally on $\Db_Y(X)$ for any closed subset $Y \subset X$.
Our next goal is to prove the following  strengthening of Proposition~\ref{prop:Restriction} when the genus of $C$ is positive. 

\begin{prop}
\label{prop:RestInjective}
    Let $f \colon X \to C$ be a morphism from a projective variety to a smooth projective curve of positive genus over an algebraically closed field $k$, and let $Y \subset X$ be a closed subset.  
Let $v \colon \rK_0(\Db_Y(X)) \to \Lambda$ be a homomorphism to a finite rank free abelian group such that $v$ is $\Pic^0(C)(k)$-equivariant when $\Lambda$ is equipped with the trivial action. 
Let $\sigma = (Z, \cP) \in \Stab_{(\Lambda, v)}(\Db_Y(X))$ be a stability condition. 
For any closed point $p \in C$, let $Y_p \subset Y$ be the fiber of $Y$ over $p$ and define: 
\begin{enumerate}
    \item $v_p \colon \rK_0(\Db_{Y_p}(X)) \to \rK_0(\Db_Y(X)) \xrightarrow{v} \Lambda$ as the natural composition; and  
    \item $\cP_p(\phi) \coloneqq \cP(\phi) \cap \Db_{Y_p}(X)$ for any $\phi \in \R$. 
\end{enumerate}
Then $\sigma_p \coloneqq (Z, \cP_p)$ is a stability condition on $\Db_{Y_p}(X)$ with respect to $(\Lambda, v_p)$, and 
the map
\begin{equation}
\label{stab-restriction-C-p}
\Stab_{(\Lambda,v)}(\Db_Y(X))\longrightarrow\Stab_{(\Lambda,v_p)}(\Db_{Y_p}(X)), \quad \sigma=(Z,\cP)\mapsto\sigma_p=(Z,\cP_p), 
\end{equation} 
is holomorphic. 

Moreover, if $\sigma_1 = (Z_1, \cP_1), \sigma_2 = (Z_2, \cP_2) \in \Stab_{(\Lambda, v)}(\Db_Y(X))$ are two stability conditions such that 
$\sigma_{1,p} = \sigma_{2,p}$ for all closed points $p \in C$, then $\sigma_1 = \sigma_2$. 
In other words, 
the maps~\eqref{stab-restriction-C-p} for varying $p$ combine into an injection 
\begin{equation*}
\Stab_{(\Lambda,v)}(\Db_Y(X)) \longrightarrow \prod_{p \in C(k)} \Stab_{(\Lambda,v_p)}(\Db_{Y_p}(X)). 
\end{equation*} 
\end{prop}

\begin{rema}\label{remark-Pic0-equivariant}
    Again by Remark~\ref{remark-numerical-v-G-equivariant}, the condition that $v$ is $\Pic^0(C)(k)$-equivariant is automatically satisfied if $v$ is numerical. 
\end{rema}

The first part of the proposition follows directly from Proposition~\ref{prop:Restriction} by considering the composition 
\begin{equation*} 
a \colon X \xrightarrow{f} C \xrightarrow{\mathrm{AJ}} \Jac(C), 
\end{equation*} 
where the second morphism is the Abel--Jacobi map, which is an embedding by our assumption that the genus of $C$ is positive. 
The following lemma is the key ingredient for the proof of the second part of the proposition. 

\begin{lemm}\label{lem:Tensor}
Let $\sigma = (Z, \cP) \in \Stab_{(\Lambda, v)}(\Db_Y(X))$ be a stability condition as in the statement of Proposition~\ref{prop:RestInjective}. 
Then for any object $E \in \cA \coloneqq \cP((0,1])$ and closed point $p \in C$, if we set $E_p = E \otimes f^*k(p)$ where $k(p)$ is the skyscraper sheaf of $p$, 
then $
E_p  \in \cP((0,2])$. 
\end{lemm}

\begin{proof}
Let $\cH^{i}_{\cA}$ denote the cohomology functors with respect to t-structure underlying $\sigma$ with heart $\cA$. 
The statement of the lemma is equivalent to the assertion that the cohomology objects $\cH^i_{\cA}(E_p)$ vanish for all $i \neq 0,-1$. 

Pulling back the exact sequence 
\[
0\longrightarrow\cO_C\longrightarrow\cO_C(p)\longrightarrow k(p) \longrightarrow 0
\]
on $C$ to $X$ via $f$ and tensoring by $E$, we get a distinguished triangle
\begin{equation*}
E\longrightarrow E\otimes f^*\cO_C(p)\longrightarrow E_p. 
\end{equation*} 
As $E \in \cA$, the long exact sequence on cohomology associated to the above triangle gives an isomorphism  
\[
\cH_{\cA}^i(E\otimes f^*\cO_C(p))\cong\cH_{\cA}^i(E_p) 
\]
for all $i\neq0,-1$. 
Theorem~\ref{theorem-invariant-stability} implies that tensoring by any $L \in \Pic^0(C)(k)$ commutes with the cohomology functors $\cH_{\cA}^i$, so for any closed point $q \in C$ we find 
\[
\cH_{\cA}^i(E\otimes f^*\cO_C(p)) \cong \cH_{\cA}^i(E\otimes f^*\cO_C(q)) \otimes f^*\cO_C(p-q)
\]
for all $i\in\Z$. 
Combined with the isomorphism above, by taking support we find that 
\begin{equation*} 
\Supp(\cH_{\cA}^i(E_p)) = \Supp(\cH_{\cA}^i(E_q)) 
\end{equation*} 
for all closed points $q \in C$ and $i \neq 0, -1$. 

On the other hand, 
applying  Lemma~\ref{lem:HNfactors} to $a \colon X \to C \hookrightarrow \Jac(C)$, we deduce that $\Supp(\cH_{\cA}^i(E_p))\subset Y_p$ 
and $\Supp(\cH_{\cA}^i(E_q)) \subset Y_q$
for all $i\in\Z$. 
Taking $q \neq p$ in the above analysis, we then find that the object $\cH_{\cA}^i(E_p)$ is supported on different fibers and hence must vanish for all $i \neq 0,-1$. 
\end{proof}

\begin{proof}[Proof of Proposition~\ref{prop:RestInjective}]
As we already noted, the first part of the proposition follows by applying Proposition~\ref{prop:Restriction} to the composition $a \colon X \to C \hookrightarrow \Jac(C)$. 

If the image of $f$ is a point $p_0$, then $Y = Y_{p_0}$ and the second part of the proposition is obvious. 
Thus we may assume that $f$ is surjective and hence flat. 
In this case, for any $E \in \Db(X)$, by base change the object $E_p \coloneqq E \otimes f^*k(p)$ can also be described as $E_p = i_{p*}i_p^*E$, where $i_p \colon X_p \to X$ is the inclusion of the fiber of $X$ over $p \in C$. 

Let $\sigma_1 = (Z_1, \cP_1)$ and $\sigma_2 = (Z_2, \cP_2)$ be as in the second part of the proposition. 
Note that given a stability condition $\sigma = (Z, \cP) \in \Stab_{(\Lambda, v)}(\Db_Y(X))$ and a closed point $p \in C$, by definition $\sigma_p = (Z, \cP_p)$, so the central charge $Z$ is determined by $\sigma_p$; in particular, our assumption $\sigma_{1,p} = \sigma_{2,p}$ implies that $Z_1 = Z_2$. 
Thus, by Lemma~\ref{lem:SameCentralCharge} to show that $\sigma_1 = \sigma_2$ it suffices to check that $\cP_2((0,1]) \subset \cP_1((-1,2])$. 
In turn, by Lemma~\ref{lem:DistOfHearts}, for this it suffices to check that for any $0 \neq E \in \cP_1((0,1])$, there exist objects
\begin{equation*}
E^{-} \in \cP_1((-\infty,2]) \cap \cP_2((-\infty, 2]) \quad \text{and} \quad 
E^+ \in \cP_{1}((-1,\infty)) \cap \cP_2((-1, \infty))
\end{equation*}
such that $\Hom(E, E^{-})$ and $\Hom(E^+,E)$ are nonvanishing. 

If $E \in \cP_1((0,1])$ is nonzero, there exists a closed point $p \in C$ such that $E_p \in \Db_{Y_p}(X)$ is nonzero. 
We claim that we may take $E^- = E_p$ and $E^+ = E_p[-1]$ above. 
By Lemma~\ref{lem:Tensor}, we have $E_p \in \cP_1((0,2])$. 
By assumption $\sigma_{1,p} = \sigma_{2,p}$, so for any interval $I$ the subcategories $\cP_{i,p}(I) = \cP_{i}(I) \cap \Db_{Y_p}(X)$ for $i=1,2$ coincide. Thus 
\begin{equation*}
    E_p \in \cP_1((0,2]) \cap \cP_2((0,2]) \quad \text{and} \quad 
    E_p[-1] \in \cP_1((-1,1]) \cap \cP_2((-1,1]). 
\end{equation*}
Moreover, since $E_p = i_{p*}i_p^*E$, we find that 
\begin{equation*}
    \Hom(E, E_p) = \Hom(i_p^*E, i_p^*E) \neq 0. 
\end{equation*}
Further, as $i_{p} \colon X_p \to X$ is the inclusion of a Cartier divisor and $\cO_X(X_p) \cong f^* \cO_C(p)$ restricts trivially to $X_p$, by Grothendieck duality the right adjoint of $i_{p*}$ is given by $i_p^*[-1]$ and we find that 
\begin{equation*}
    \Hom(E_p[-1], E) = \Hom(i_p^*E[-1], i_p^*E[-1]) \neq 0. 
\end{equation*}
This verifies that $E^- = E_p$ and $E^+ = E_p[-1]$ satisfy the required conditions. 
\end{proof}

\begin{rema}\label{rmk:RestInjective}
Later we will use the following slight amplification of Proposition~\ref{prop:RestInjective}. 
Let $f \colon X \to C$ be a morphism from a projective variety to a smooth projective curve of positive genus over an algebraically closed field $k$, and let $Y \subset X$ be a closed subset. 
Suppose that $\sigma_1=(Z_1,\cP_1)\in\Stab_{(\Lambda_1,v_1)}(\Db_Y(X))$ and $\sigma_2=(Z_2,\cP_2)\in\Stab_{(\Lambda_2,v_2)}(\Db_Y(X))$, where $v_1$ and $v_2$ are both $\Pic^0(C)(k)$-equivariant when $\Lambda_1$ and $\Lambda_2$ are equipped with the trivial action (which holds, for instance, if both are numerical by Remark~\ref{remark-Pic0-equivariant}), but the pairs $(\Lambda_1,v_1)$ and $(\Lambda_2,v_2)$ 
do not necessarily coincide.
If we assume that
\[
Z_1\circ v_1 = Z_2\circ v_2\colon \rK_0(\Db_Y(X))\longrightarrow\C
\]
and that there is an equality of slicings $\cP_{1,p} = \cP_{2,p}$ for all closed points $p \in C$, 
then we can still conclude that there is an equality $\cP_1 = \cP_2$ of slicings. 

Indeed, by Remark~\ref{rmk:Surjectivity}, we can first assume that $v_1$ and $v_2$ are surjective. Then, we can use Remark~\ref{rmk:CompareSupportProperty} and consider the stability conditions $\overline{\sigma}_1,\overline{\sigma}_2\in\Stab_{(\overline{\Lambda},\overline{v})}(\cD)$, where $\overline{\Lambda}$ denotes the image of $Z_1\circ v_1=Z_2\circ v_2$ in $\C$, seen as an abstract finitely generated free abelian group.
Then, $\overline{Z}_1=\overline{Z}_2$ and $\overline{\sigma}_{1,p}=\overline{\sigma}_{2,p}$, for all closed points $p\in C$.
The conclusion now follows from Proposition~\ref{prop:RestInjective}.
\end{rema}

\subsection{Stability conditions on categories supported at a point} 

For use below, we record a simple observation about stability conditions on the subcategory 
of objects supported at a closed point. 

\begin{lemm} \label{lemma-stability-DxX}
    Let $X$ be a variety with a closed point $x \in X$. 
    Let $\sigma = (Z, \cP)$ be a pre-stability condition on $\Db_x(X)$ with respect to a homomorphism $v \colon \rK_0(\Db_x(X)) \to \Lambda$. 
    If the skyscraper sheaf $k(x)$ is $\sigma$-semistable, then the slicing $\cP$ of $\sigma$ is uniquely determined by the phase $\phi_{\cP}(k(x))$; 
    explicitly, for $\phi \in \R$ we have 
    \begin{equation*}
        \cP(\phi) = 
        \begin{cases}
        \langle k(x)[n] \rangle & \text{if $\phi = \phi_{\cP}(k(x)) + n$ for some $n \in \Z$}, \\ 
        0 & \text{if $\phi \notin \phi_{\cP}(k(x)) + \Z$}, 
        \end{cases}
    \end{equation*}
    where $\langle k(x)[n] \rangle \subset  \Db_x(X)$ denotes the extension closed subcategory generated by $k(x)[n]$. 
\end{lemm}

\begin{proof}
    Let $E \in \Db_x(X)$. 
    If $E$ has cohomological amplitude in $[a,b]$ with respect to the standard t-structure on $\Db_x(X)$, then 
    \begin{equation*}
        0 = \tau^{\leq a - 1}E \xrightarrow{f_a} \tau^{\leq a}E \xrightarrow{f_{a+1}} \tau^{\leq a +1}E \to \cdots \xrightarrow{f_{b}} \tau^{\leq b} = E
    \end{equation*}
    is a filtration of $E$ such that $\mathrm{cone}(f_i) \cong \cH^i(E)[-i]$.
    Since $\cH^i(E)$ is a sheaf supported at $x$, it admits a filtration with graded pieces isomorphic to the skyscraper sheaf $k(x)$, and hence is $\sigma$-semistable of phase $\phi_{\cP}(k(x))$. 
    Thus, 
    $\cH^i(E)[-i]$ is $\sigma$-semistable of phase $\phi_{\cP}(k(x))-i$, and the above filtration of $E$ is in fact the Harder--Narasimhan filtration with respect to $\sigma$. 
    This implies the claimed formula for the category $\cP(\phi)$ of $\sigma$-semistable objects of phase $\phi$. 
\end{proof}

\begin{rema}
    There is a version of Lemma~\ref{lemma-stability-DxX} which holds without the assumption that $k(x)$ is $\sigma$-semistable. 
    Since we only need the simpler special case for our purposes in this paper, we defer the general result to the forthcoming work \cite{LMPSZ:Deformation}. 
\end{rema}

\subsection{Proof of Theorem~\ref{thm:Main}}\label{subsec:Proof}

We will prove the following slightly more precise version of Theorem~\ref{thm:Main}. 

\begin{theo}\label{thm:Main-precise}
Let $X = C_1 \times \cdots \times C_n$, where $C_1, \dots, C_n$ are smooth projective curves of positive genus over an algebraically closed field $k$. 
For $i = 1,2$, let $v_i \colon  \rK_0(X) \to \Lambda_i$ be a numerical homomorphism to a finite rank free abelian group. 
Let $\sigma_1 = (Z_1, \cP_1) \in \Stab_{(\Lambda_1, v_1)}(\Db(X))$ and 
$\sigma_2=(Z_2,\cP_2)\in\Stab_{(\Lambda_2,v_2)}(\Db(X))$ be stability conditions such that: 
\begin{enumerate}
    \item $Z_1\circ v_1=Z_2\circ v_2$; and 
    \item for any closed point $x \in X$, there is an equality of phases $\phi_{\cP_1}(k(x)) = \phi_{\cP_2}(k(x))$. 
\end{enumerate} 
Then there is an equality of slicings $\cP_1 = \cP_2$. 
\end{theo}

Recall that by \cite{FLZ:Albanese}, for $X$ as above and any stability condition $\sigma \in \Stab_{(\Lambda, v)}(\Db(X))$, the skyscraper sheaves $k(x)$ of closed points $x \in X$ are all $\sigma$-stable of the same phase. 

\begin{proof}
Assume for sake of contradiction that $\cP_1 \neq \cP_2$. 
We claim that there exist closed points $p_j \in C_j$ for $j=1,\dots,n$ such that for each $i = 1,2$ and each $j$, the stability condition 
$\sigma_{i}$ restricts to a stability condition 
$\sigma_{i, p_1, \dots, p_j} = (Z_i, \cP_{i,p_1, \dots, p_j}) \in \Stab_{(\Lambda_i, v_{i, p_1, \dots, p_j})}(\Db_{Y_j}(X))$, where: 
\begin{itemize}
\item $Y_j \coloneqq \{p_1\}\times\dots\times\{p_j\}\times C_{j+1}\times\dots\times C_{n}$, 
\item $v_{i, p_1, \dots, p_j} \colon \rK_0(\Db_{Y_j}(X)) \to \rK_0(\Db(X)) \xrightarrow{ v_{i} } \Lambda_i$ is the natural composition,  
\item $\cP_{i, p_1, \dots, p_j}(\phi) = \cP_i(\phi) \cap \Db_{Y_j}(X)$ for all $\phi \in \R$, and 
\item $\cP_{1, p_1, \dots, p_j} \neq \cP_{2, p_1, \dots, p_j}$. 
\end{itemize}
Indeed, we can inductively choose the points $p_j \in C_j$ by applying the amplification of Proposition~\ref{prop:RestInjective} from Remark~\ref{rmk:RestInjective} to the $j$th projection $X \to C_j$. 

Note that $Y_n$ is the closed point $x = (p_1, \dots, p_n) \in X$. 
Moreover, for $i = 1,2$, the skyscraper sheaf $k(x)$ is $\sigma_i$-stable of phase $\phi_{\cP_i}(k(x))$ by construction. 
Since $\phi_{\cP_1}(k(x)) = \phi_{\cP_2}(k(x))$, we conclude by Lemma~\ref{lemma-stability-DxX} that 
$\cP_{1, p_1, \dots, p_n} = \cP_{2, p_1, \dots, p_n}$, contradicting the fact that these slicings are distinct by construction. 
\end{proof}


\section{Hilbert schemes of points on surfaces}\label{subsec:Hilbert}

The goal of this section is to prove our second main result, Theorem~\ref{theorem-Hilbn}. 
The key ingredient is the construction from~\cite{LiuY:Products} of stability conditions on the product of a variety with a curve, which we review in Section~\ref{subsec:Products}. 
In Section~\ref{subsec:StabilityProductCurves}, we prove a strengthening (Theorem~\ref{thm:LiuY+}) of the support property for the resulting stability conditions on products of curves. 
This is needed for the proof of Theorem~\ref{theorem-Hilbn} in Section~\ref{section-proof-theorem-Hilbn}.

Throughout this section, we work over an algebraically closed field $k$, except in Section~\ref{section-proof-theorem-Hilbn} where we additionally assume that $\mathrm{char}(k)=0$ in order to apply the BKR equivalence. 

\subsection{Stability conditions on products}\label{subsec:Products}

In~\cite{LiuY:Products}, Bridgeland stability conditions have been constructed on the product of a variety with a curve, given a stability condition on the variety itself.
We briefly recall this construction.

Let $Y$ be a smooth projective variety over $k$ and let $\sigma=(Z,\cA)$ be a stability condition on $\Db(Y)$ with respect to a fixed pair $(\Lambda,v)$, such that the image of $Z$ is in $\Q\oplus\Q i$; in particular, by~\cite[Lemma~12.12]{BLMNPS:Family}, the heart $\cA$ is noetherian.
Let $C$ be a smooth projective curve over $k$ of genus $g\geq0$, and fix an ample line bundle $\cO_C(1)$ on $C$.
To construct stability conditions on $\Db(Y\times C)$, the idea in~\cite{LiuY:Products} is to use the base changed t-structure constructed in~\cite[Theorem 2.6.1]{AP:Sheaf} (see also~\cite{Pol:Family}). 
Namely, starting from the noetherian heart $\cA$, we obtain an induced bounded t-structure on $\Db(Y\times C)$, with noetherian heart given by
\[
\cA_C\coloneqq\left\{ E\in\Db(Y\times C) \st p_{Y*}(E\otimes p_C^*\cO_C(m))\in\cA \text{ for }m\gg0\right\},
\]
where $p_Y$ and $p_C$ denote the two projections.
Then, as in~\cite[Section 4]{LiuY:Products}, for $E\in\Db(Y\times C)$ we write
\begin{equation}\label{eq:Zsigma}
Z(p_{Y*}(E\otimes p_C^*\cO_C(m+g-1)))=a(E)m+b(E)+i(c(E)m+d(E)),
\end{equation}
for some group homomorphisms $a,b,c,d\colon\mathrm \rK_0(Y\times C)\to\R$.\footnote{We use a different normalization than~\cite{LiuY:Products}: we added the linear factor $g-1$ in~\eqref{eq:Zsigma}, as it simplifies the computations later. 
This keeps the functions $a$ and $c$ the same as in \cite{LiuY:Products}, but changes the functions $b$ and $d$ by the shifts $b\mapsto b+(1-g)a$ and $d\mapsto d+(1-g)c$. This does not affect any of the proofs.}
We consider the finite rank free abelian groups 
\begin{equation}\label{eq:LambdaC}
\Lambda_C\coloneqq\Lambda\oplus\Lambda\qquad  \text{and} \qquad \widetilde{\Lambda}_C\coloneqq\Lambda \oplus \left(\Lambda/\Ker(Z)\right),
\end{equation}
and define
\begin{equation}\label{eq:vC}
v_C=(v_{C}',v_{C}'')\colon \rK_0(Y\times C)\longrightarrow\Lambda_C,    
\end{equation}
where
\begin{equation}\label{eq:vLiu}
\begin{split}
& v_{C}'(E)\coloneqq v(p_{Y*}(E\otimes p_C^*\cO_C(m)))-v(p_{Y*}(E\otimes p_C^*\cO_C(m-1))),\\
& v_{C}''(E)\coloneqq v(p_{Y*}(E\otimes p_C^*\cO_C(m)))-(m+1-g)\cdot v_{C}'(E),
\end{split}
\end{equation}
for $E\in\Db(Y\times C)$, which are independent of $m\in\Z$.
We also define
\[
\widetilde{v}_C\colon \rK_0(Y\times C)\longrightarrow\widetilde{\Lambda}_C
\]
as the composition of $v_C$ with the projection onto $\widetilde{\Lambda}_C$.

For every $t\in\R_{>0}$ and $\beta\in\R$, we define $^{\prime\!}Z_C^{t,\beta}\colon\widetilde{\Lambda}_C\to\C$ by
\begin{equation*}
^{\prime\!}Z_C^{t,\beta}(E)\coloneqq a(E)t-d(E)+c(E)\beta+ic(E)t,
\end{equation*}
for $E\in\Db(Y\times C)$.
By~\cite[Theorem 3.3 and Lemma 4.3]{LiuY:Products}\footnote{In \cite{LiuY:Products} there is no parameter $\beta$, but again it does not change anything in the proofs.}, the pair $(^{\prime\!}Z_C^{t,\beta},\cA_C)$ is a \emph{weak stability condition} with respect to $(\widetilde{\Lambda}_C,\widetilde{v}_C)$ (see e.g.~\cite[Section 14.1]{BLMNPS:Family} for the basic definitions about weak stability conditions). 
In particular, the subcategories 
\begin{align*}
\cT_C^{t,\beta} & \coloneqq 
\left\{E\in\cA_C ~ \left| ~ \begin{array}{c}
\text{for every surjection }E\twoheadrightarrow F \text{ in }\cA_C,\, \text{either} \\ 
 c(F)=0, \text{ or } c(F)\neq 0\text{ and }\Re {^{\prime\!}Z}_C^{t,\beta}(F)<0
\end{array}
\right. \right\}, \\
\cF_C^{t,\beta} & \coloneqq 
\left\{E\in\cA_C ~ \left| ~ \begin{array}{c} 
\text{for every injection }F\hookrightarrow E\text{ in }\cA_C, \\ 
c(F)\neq 0\text{ and }\Re {^{\prime\!}Z}_C^{t,\beta}(F)\geq0 
\end{array}
\right. \right\},
\end{align*}
form a torsion pair in $\cA_{C}$, and we can define the tilted heart
\begin{equation}\label{eq:cALiu}
\cA_C^{t,\beta}\coloneqq\langle \cT_C^{t,\beta},\cF_C^{t,\beta}[1]\rangle
\end{equation}
as the extension closure of the subcategories $\cT_C^{t,\beta}$ and $\cF_C^{t,\beta}[1]$ in $\Db(Y \times C)$. 
For every $s\in\R_{>0}$, we define a new central charge $Z^{s,t,\beta}_C\colon\widetilde{\Lambda}_C\to\C$ by
\begin{equation}\label{eq:ZLiu}
Z^{s,t,\beta}_C(E)=b(E)+sc(E)-\beta a(E)+i(d(E)-ta(E)-\beta c(E))
\end{equation}
for $E\in\Db(Y\times C)$.

\begin{theo}[{\cite[Theorem 5.9]{LiuY:Products}}]\label{thm:LiuY}
For every $t \in \R_{>0}$, $\beta \in \R$, and $s \in \R_{>0}$, the pair $\sigma_C^{s,t,\beta}\coloneqq(Z^{s,t,\beta}_C,\cA_C^{t,\beta})$ is a 
stability condition on $\Db(Y\times C)$ with respect to $(\widetilde{\Lambda}_C,\widetilde{v}_C)$. 
\end{theo}

Theorem~\ref{thm:LiuY} allows us to inductively construct stability conditions on products of curves.
We make this explicit in the next section, further expanding~\cite[Corollary 1.2]{LiuY:Products} to show the support property with respect to a larger lattice.

\subsection{The support property on products of curves}\label{subsec:StabilityProductCurves}

Let $\{C_n\}_{n\geq1}$ be a set of smooth projective curves over $k$, with genera $\{g_n\geq0\}_{n\geq1}$.
We fix ample line bundles $\{\cO_{C_n}(1)\}_{n\geq1}$ of degree~1 on each $C_n$.
Let
\[
X_n\coloneqq C_1\times\dots\times C_n.
\]
We will abuse notation and denote by $h_1,h_2,\dots$ the first numerical Chern classes of both $\cO_{C_1}(1),\cO_{C_2}(1),\dots$ and their pullbacks to $X_n$. 
We denote by $h_{1,\dots,n}$ the first numerical Chern class of the line bundle $\cO_{C_1}(1)\boxtimes\dots\boxtimes\cO_{C_n}(1)$, or equivalently
\[
h_{1,\dots,n} = h_1+\dots+h_n. 
\]

We define data $(\Lambda_n,v_n)$ inductively as follows.
We set
\[
\Lambda_1 \coloneqq \Z^{\oplus 2} 
\]
and
\[
v_1\colon \rK_0(X_1)\longrightarrow\Lambda_1 , \qquad E\longmapsto\left(\rk(E),\deg(E)\right).
\]
Thinking of $X_n$ as $X_{n-1}\times C_n$, we define
\[
\Lambda_n\coloneqq(\Lambda_{n-1})_{C_n}=\Lambda_{n-1}\oplus\Lambda_{n-1}
\]
in the notation~\eqref{eq:LambdaC}, and
\[
v_n=(v_{n-1})_{C_n}=(v_n',v_n'')\colon \rK_0(X_n)\longrightarrow\Lambda_n
\]
in the notation~\eqref{eq:vC}.

\begin{rema}
\label{remark-Lambdan}
    Sometimes it is convenient to 
    think of the recursive definition of $\Lambda_n$ as the formula $\Lambda_n = \Lambda_{n-1} \otimes \Lambda_1$. 
    Expanding, we can view $\Lambda_n$ as the tensor product 
    \begin{equation*}
        \Lambda_n \cong \Lambda_1^{\otimes n} \cong \bigoplus_{\substack{0 \leq \ell \leq n \\ 1 \leq j_1 < \cdots < j_{\ell} \leq n}} \Z^{\otimes \ell} \otimes \Z^{\otimes n-\ell} \cong \Z^{ \oplus 2^n}, 
    \end{equation*}
    where the second isomorphism is the ``binomial expansion'' of $(\Z \oplus \Z)^{\otimes n}$. 
\end{rema}

In these terms, we can explicitly describe the map $v_{n} \colon \rK_0(X_n) \to \Lambda_{n}$ from above.  

\begin{lemm}\label{lem:GRR}
For any $E \in \Db(X_n)$, under the isomorphism 
\begin{equation*}
    \Lambda_{n-1} \cong \bigoplus_{\substack{0 \leq \ell \leq n-1 \\ 1 \leq j_1 < \cdots < j_{\ell} \leq n-1}} \Z
\end{equation*}
from above, the components of $v'_n(E)$ and $v''_n(E)$ are given by
\begin{align}
&v_n'(E)=\left(h_{j_1}\dots h_{j_\ell}h_n\ch_{n-\ell-1}(E)\right)_{1\leq j_1<\dots<j_\ell \leq n-1}\label{eq:GRR1} ,\\
    &v_n''(E)=\left(h_{j_1}\dots h_{j_\ell}\ch_{n-\ell}(E)\right)_{1\leq j_1<\dots<j_\ell\leq n-1}\label{eq:GRR2}.   \end{align}
Equivalently, under the isomorphism 
\begin{equation*}
    \Lambda_{n} \cong 
    \bigoplus_{\substack{0 \leq \ell \leq n \\ 1 \leq j_{1} < \cdots < j_{\ell} \leq n}} \Z, 
\end{equation*}
the components of $v_n(E)$ are given by 
\begin{equation*}
    v_n(E) = (h_{j_1} \dots h_{j_{\ell}} \ch_{n-\ell}(E))_{1 \leq j_1 < \cdots < j_{\ell} \leq n}. 
\end{equation*}
\end{lemm}

\begin{proof}
The proof is by induction on $n$. 
The base case $n = 1$ holds by definition. 
For general $n$, by using the explicit expression from~\eqref{eq:vLiu} (with $m=0$), we have
\[
v_n'(E)=v_{n-1}\left(E|_{X_{n-1}}\right),
\]
which by induction implies the formula~\eqref{eq:GRR1}. 
To prove~\eqref{eq:GRR2}, we first note that again by~\eqref{eq:vLiu} (with $m=0$), we have 
\begin{equation*}
    v''_n(E) = v_{n-1}(p_{X_{n-1}*}(E)) - (1 -g_n) v'_n(E). 
\end{equation*}
On the other hand, by Grothendieck--Riemann--Roch we have
\begin{equation}\label{eq:GRR3}
\ch(p_{X_{n-1}*}(E))=p_{X_{n-1}*}\left(\ch(E)\td_{C_n}\right)=p_{X_{n-1}*}\left(\ch(E)(1+(1-g_n)h_n)\right).     
\end{equation}
Together with the above formula for $v''_n(E)$, by induction this implies the formula~\eqref{eq:GRR2}. 
\end{proof}

Next we explicate the central charges for the stability conditions we will consider. 
We fix $(B,\omega)\in\Q\times\Q_{>0}$ and define
\[
Z_n^{B,\omega}\colon\Lambda_n\longrightarrow\C
\]
to be the function so that 
\begin{equation}
\label{ZBomegan}
    Z^{B,\omega}_n(v_n(E)) = -\int_{X_n} e^{-(B+i\omega)h_{1,\dots,n}}\ch(E). 
\end{equation}
Note that by Lemma~\ref{lem:GRR}, there is indeed a natural function $Z_n^{B,\omega} \colon \Lambda_n \to \C$ determined by this property. 
When $n=1$, we get
\[
Z^{B,\omega}_1(E)=-\deg(E)+B\rk(E)+i\omega\rk(E).
\]
For $n \geq 2$, $Z_{n}^{B,\omega}$ agrees with the central charge defined inductively by the construction in Section~\ref{subsec:Products}, for suitable parameters: 

\begin{lemm}\label{lem:Zbw}
For $n \geq 2$, let $Z_{C_n}^{\omega,\omega,B}$ be the central charge defined in~\eqref{eq:ZLiu} for $s = t = \omega$ and $\beta = B$, using $Z = Z^{B,\omega}_{n-1}$ as the input central charge. 
Then we have $Z_{C_n}^{\omega,\omega,B}=Z^{B,\omega}_n$.
\end{lemm}

\begin{proof}
For any $E\in\Db(X_{n-1}\times C_n)$, by~\eqref{eq:GRR3} we have
\[
\begin{split}
    Z^{B,\omega}_{n-1}(p_{X_{n-1}*}(E \, \otimes  & \, p_{C_n}^*\cO_{C_n}(m+g_n-1))) \\
    &= -\int_{X_{n-1}}e^{-(B+i\omega)h_{1,\dots,n-1}}\ch(p_{X_{n-1}*}(E\otimes p_{C_n}^*\cO_{C_n}(m+g_n-1)))\\
      &= -\int_{X_{n-1}}e^{-(B+i\omega)h_{1,\dots,n-1}}\,p_{X_{n-1}*}\left(\ch(E)(1+mh_n)\right)\\
      &=-\int_{X_n} e^{-(B+i\omega)h_{1,\dots,n-1}}\ch(E)(1+m h_n).
\end{split}
\]
In the notation of~\eqref{eq:Zsigma}, we deduce that
\[
a(E)+ic(E)=-\int_{X_n} e^{-(B+i\omega)h_{1,\dots,n-1}}h_n\ch(E)
\]
and
\[
b(E)+id(E)=-\int_{X_n} e^{-(B+i\omega)h_{1,\dots,n-1}}\ch(E).
\]
Writing~\eqref{eq:ZLiu} with $s=t=\omega$ and $\beta=B$, we get
\[
\begin{split}
Z^{\omega,\omega,B}_C(E)&=-\left(B+i\omega\right) \left(a(E)+ic(E)\right) + \left(b(E)+id(E)\right)\\
    &=-\int_{X_n} e^{-(B+i\omega)h_{1,\dots,n-1}}(1-(B+i\omega)h_n)\ch(E)
\end{split}
\]
For $\ell \geq 0$ we have the relation 
\begin{equation*}
    h_{1,\dots,n}^\ell = h_{1,\dots,n-1}^\ell + \ell h_{1,\dots,n-1}^{\ell-1} h_n, 
\end{equation*} 
which implies 
\begin{equation*}
    e^{-(B+i\omega)h_{1,\dots,n}} = e^{-(B+i\omega)h_{1,\dots,n-1}}(1-(B+i\omega)h_n). 
\end{equation*}
Plugging into the previous equation for $Z^{\omega, \omega, B}_C(E)$, we find 
\begin{equation*}
   Z^{\omega,\omega,B}_C(E) =  -\int_{X_n} e^{-(B+i\omega)h_{1,\dots,n}}\ch(E) = Z_n^{B,\omega}(E), 
\end{equation*}
as claimed. 
\end{proof}

For $(B, \omega) \in \Q \times \Q_{>0}$, we now inductively define a stability condition $\sigma_n = (Z_{n}^{B, \omega}, \cA_n)$ on $\Db(X_n)$ whose central charge is the one considered above. 
For $n = 1$, we take $\cA_1 = \Coh(X_1)$ to be the category of coherent sheaves, so that $\sigma_1 = (Z_1^{B,\omega}, \cA_1)$ is the standard stability condition on the curve $X_1$. 
For $n \geq 2$, we define $\cA_n$ to be the tilted heart given by~\eqref{eq:cALiu} for $t = \omega$, $\beta = B$, using $\sigma_{n-1} = (Z_{n-1}^{B,\omega}, \cA_{n-1})$ as the input stability condition. 
Note that the assumption that $(B,\omega)$ is rational guarantees that $Z_{n-1}^{B,\omega}$ takes values in $\Q \oplus \Q i$, so that the construction from Section~\ref{subsec:Products} can indeed be applied. 
Moreover, Theorem~\ref{thm:LiuY} (combined with Lemma~\ref{lem:GRR} and Lemma~\ref{lem:Zbw}) shows that $\sigma_n^{B,\omega} = (Z_n^{B,\omega}, \cA_{n})$ is a stability condition. 

However, Theorem~\ref{thm:LiuY} only shows that $\sigma_n^{B,\omega} = (Z_n^{B,\omega}, \cA_{n})$ is a stability condition with respect to a quotient of the group $\Lambda_n$. 
This quotient is not preserved by many group actions --- including the symmetric group, relevant to our applications --- so it is not sufficient for our purposes. 
For this reason, we need the following extension of the support property to the whole group $\Lambda_n$. 

\begin{theo}\label{thm:LiuY+}
Assume that the curves $C_1, \dots, C_n$ all have positive genus. 
Then for any $(B,\omega)\in\Q\times\Q_{>0}$, the pair $\sigma^{B,\omega}_n=(Z^{B,\omega}_n,\cA_n)$ defined above is a stability condition on $\Db(X_n)$ with respect to $(\Lambda_n,v_n)$. 
Moreover, for any closed point $x \in X_n$, the skyscraper sheaf $k(x)$ is $\sigma^{B,\omega}_n$-stable of phase $1$. 
\end{theo}

\begin{proof}
We proceed by induction on $n$.
The base case $n=1$ is clear. 
Let $n\geq2$. 
We introduce the following notation: 
for $r=1,\dots,n$, we set
\[
X_{\widehat r}\coloneqq C_1\times \dots  \times \widehat{C_r} \times \dots\times C_n,
\]
the product of $n-1$ curves, obtained by omitting the factor $C_r$ in $X_n$.
By the inductive assumption, we have stability conditions $\sigma_{\widehat r}^{B,\omega}$ on $\Db(X_{\widehat r})$ with respect to $(\Lambda_{\widehat r},v_{\widehat r})$, 
where $(\Lambda_{\widehat r},v_{\widehat r})$ is defined analogously to $(\Lambda_{n-1},v_{n-1})$, using the curves $C_i$ for $1 \leq i \leq n, i \neq r$, in place of $1 \leq i \leq n-1$;  
e.g. $(\Lambda_{\widehat n},v_{\widehat n})=(\Lambda_{n-1},v_{n-1})$. 
Moreover, by the inductive assumption, 
skyscraper sheaves are $\sigma_{\widehat{r}}$-stable of phase~1 for all $r$. 

By Theorem~\ref{thm:LiuY}, each stability condition $\sigma_{\widehat r}^{B,\omega}$ gives a stability condition
\[
\sigma_{n,\widehat r}^{B,\omega}=(Z_{n, \widehat r}^{B,\omega},\cA_{n,\widehat r})
\]
on $\Db(X_n)$, with respect to the lattice prescribed in~\eqref{eq:LambdaC}: 
\[
\widetilde{\Lambda}_{n,\widehat r}\coloneqq(\widetilde{\Lambda_{\widehat r}})_{C_{r}}=\Lambda_{\widehat r}\oplus\left(\Lambda_{\widehat r}/\Ker(Z_{\widehat r}^{B,\omega})\right).
\]
Let us denote by $p_{\widehat r}$ the projection
\[
p_{\widehat r}\colon\Lambda_n\longtwoheadrightarrow\widetilde{\Lambda}_{n,\widehat r}. 
\]
Note that, by formula~\eqref{eq:GRR1} from Lemma~\ref{lem:GRR}, the composition
\[
q_{\widehat r}\colon\Lambda_n\longtwoheadrightarrow \widetilde{\Lambda}_{n,\widehat r}\longtwoheadrightarrow\Lambda_{\widehat r}
\]
satisfies
\[
q_{\widehat r}(v_n(E))=\left(h_{j_1}\dots h_{j_\ell}h_r\ch_{n-\ell-1}(E)\right)_{\substack{{1\leq j_1<\dots<j_\ell\leq n}\\ {j_1,\dots,j_\ell \neq r}}}.
\]
Since furthermore $p_{\widehat r}(v_n(k(x)))\neq0$ for $x\in X_n$, we deduce that
\begin{equation}\label{eq:OPT}
\bigcap_{r=1,\dots,n} \ker(p_{\widehat r}) = \{0\}.   
\end{equation}

To conclude the proof, the key point is that the central charges of the $\sigma_{n,\widehat r}^{B,\omega}$ are compatible over $r$. Namely, by Lemma~\ref{lem:Zbw}, for any $r$ we have 
\[
Z_n^{B,\omega} = Z_{n, \widehat r}^{B,\omega} \circ p_{\widehat r}.
\]
Moreover, skyscraper sheaves are $\sigma_{n,\widehat r}^{B,\omega}$-stable and by~\eqref{eq:cALiu} still of phase $1$. 
Hence, we can apply Theorem~\ref{thm:Main-precise} 
to deduce that the hearts $\cA_{n,\widehat{r}}$ of the stability conditions $\sigma_{n,\widehat r}^{B,\omega}$ coincide for all $r=1,\dots,n$. 
Combined with~\eqref{eq:OPT}, this shows that all the assumptions of~\cite[Lemma~7.5]{HaidenSung:EEE} (which was in turn inspired by~\cite[Lemma 3.19]{OPT}) are satisfied\footnote{Technically, \cite[Lemma 7.5]{HaidenSung:EEE} is stated in the case where the numerical Grothendieck group is the enlarged lattice to which we want to lift the support property; here we instead use the statement for $\Lambda_n$, but the same proof works directly in our situation.} and we deduce that the support property holds for $\sigma^{B,\omega}_n$ with respect to $(\Lambda_n,v_n)$, as we wanted.
\end{proof}

\subsection{Proof of Theorem~\ref{theorem-Hilbn}}
\label{section-proof-theorem-Hilbn}

\eqref{Hilbn-C1C2} Let $C_1$ and $C_2$ be smooth projective curves of positive genus over an algebraically closed field $k$ of characteristic $0$, as in part~\eqref{Hilbn-C1C2} of Theorem~\ref{theorem-Hilbn}. 
Let $S = C_1 \times C_2$ be their product, and consider the $n$-fold product of this surface, $X_{2n} = S^n$. 
Fix degree $1$ line bundles $\cO_{C_1}(1)$ and $\cO_{C_2}(1)$ on $C_1$ and $C_2$. 
We use these line bundles in the construction of Theorem~\ref{thm:LiuY+}, which for any $(B, \omega) \in \Q \times \Q_{>0}$ gives a stability condition $\sigma^{B,\omega}_{2n} = (Z^{B,\omega}_{2n}, \cA_{2n})$ on $\Db(X_{2n})$ with respect to $(\Lambda_{2n}, v_{2n})$. 

The group $\fS_n$ acts on $X$ permuting the $n$ factors of $S$. 
By definition (see Remark~\ref{remark-Lambdan}), we can write $\Lambda_{2n} = (\Lambda_1 \otimes \Lambda_1)^{\otimes n}$, where $\Lambda_1 = \Z \oplus \Z$. 
If we endow $\Lambda_{2n}$ with the $\fS_n$-action swapping the $n$ factors of $\Lambda_1 \otimes \Lambda_1$, then it follows from 
Lemma~\ref{lem:GRR} that $v_{2n} \colon \rK_0(X_{2n}) \to \Lambda_{2n}$ is $\fS_{n}$-equivariant. 
Moreover, it follows from the defining  formula~\eqref{ZBomegan} that $Z_{2n}^{B,\omega}$ is $\fS_n$-invariant. 
Thus, by Corollary~\ref{cor:Invariance}, we find that the stability condition $\sigma^{B,\omega}_{2n} \in \Stab_{(\Lambda_{2n}, v_{2n})}(\Db(X_{2n}))$ is $\fS_n$-invariant. 
Therefore, we may apply Theorem~\ref{theorem-G-invariant-stability} to obtain an induced stability condition 
$(\sigma_{2n}^{B,\omega})^{\fS_n}$ on $\Db([X_{2n}/\fS_n])$ with respect to a homomorphism $v_{2n}^{\fS_n} \colon \rK_0([X_{2n}/\fS_{n}]) \to \Lambda_{2n}^{\fS_n}$. 

On the other hand, by the BKR correspondence for Hilbert schemes \cite{BKR:MukaiMcKay,Hai:Hilb}, we have a derived equivalence
\[
\Db\left(\Hilb^n(S)\right)\cong\Db\left([X_{2n}/\fS_n]\right).
\]
Since we have already constructed a stability condition on the right-hand side, this completes the proof of part~\eqref{Hilbn-C1C2} of Theorem~\ref{theorem-Hilbn}. 

\eqref{Hilbn-Kum} Now let $E_1$ and $E_2$ be elliptic curves as in part~\eqref{Hilbn-Kum} of Theorem~\ref{theorem-Hilbn}. 
As above, if we choose degree $1$ line bundles $\cO_{E_1}(1)$ and $\cO_{E_2}(1)$ and set $A = E_1 \times E_2$ and $X_{2n} = A^n$, then we obtain for any $(B,\omega) \in \Q \times \Q_{>0}$ a stability condition $\sigma_{2n}^{B,\omega}$ on $\Db(X_{2n})$ with respect to $(\Lambda_{2n}, v_{2n})$. 
This time, instead of $\fS_n$, we consider the action of the group $\bm{\mu}_2^n \rtimes\fS_n$ on $X_{2n}$, where $\fS_n$ acts by permuting the $n$ factors of $A$ and the $i$th factor of $\bm{\mu}_2$ acts by $-1$ on the $i$th factor of $A$. 
We endow $\Lambda_{2n} = (\Lambda_1 \otimes \Lambda_1)^{\otimes n}$ with the $\bm{\mu}_2^n \rtimes\fS_n$ action where $\fS_n$ acts by permuting the $n$ factors of $\Lambda_1 \otimes \Lambda_1$ and $\bm{\mu}_2^n$ acts trivially.  Then it follows from Lemma~\ref{lem:GRR} that $v_{2n} \colon \rK_0(X_{2n}) \to \Lambda_{2n}$ and is $\bm{\mu}_2^n \rtimes\fS_{n}$-equivariant and from~\eqref{ZBomegan} that $Z_{2n}^{B,\omega}$ is $\bm{\mu}_2^n \rtimes\fS_n$-invariant. 
Hence, as above, we obtain an induced stability condition 
$(\sigma_{2n}^{B, \omega})^{\bm{\mu}_2^n \rtimes\fS_n}$ on $\Db([X_{2n}/\bm{\mu}_2^n \rtimes\fS_n])$ with respect to a homomorphism $v_{2n}^{\bm{\mu}_2^n \rtimes\fS_n} \colon \rK_0([X_{2n}/\bm{\mu}_2^n \rtimes\fS_n]) \to \Lambda_{2n}^{\bm{\mu}_2^n \rtimes\fS_n} = \Lambda_{2n}^{\fS_n}$. 

On the other hand, if $S = \Kum(A)$ is the Kummer K3 surface of $A$, then again by the BKR correspondence, we have derived equivalences 
\begin{equation*}
    \Db(S) \simeq \Db([A/\bm{\mu}_2]) 
    \quad \text{and} \quad 
    \Db(\Hilb^n(S)) \simeq \Db([S^{n}/\fS_n]). 
\end{equation*} 
These combine to an equivalence 
\begin{equation*} 
\Db(\Hilb^n(S)) \simeq \Db([X_{2n}/\bm{\mu}_2^n \rtimes\fS_n]). 
\end{equation*} 
Since we have already constructed a stability condition on the right-hand side, this completes the proof of part~\eqref{Hilbn-Kum} of Theorem~\ref{theorem-Hilbn}. \qed

\begin{rema}\label{rmk:InvariantLatticeHilb}
As the above proof shows, in the situation of Theorem~\ref{theorem-Hilbn}, the stability conditions we produced are with respect to a homomorphism $v \colon \rK_0(\Hilb^n(S)) \to \Lambda_{2n}^{\fS_n}$.  
\end{rema}


\end{document}